\documentclass[a4paper,11pt,reqno]{article}

\usepackage[utf8]{inputenc}
\usepackage[T1]{fontenc}
\usepackage{lmodern}
\usepackage[english]{babel}
\usepackage{microtype}

\usepackage{amsmath,amssymb,amsfonts,amsthm}
\usepackage{mathtools,accents}
\usepackage{mathrsfs}
\usepackage{aliascnt}
\usepackage{braket}
\usepackage{bm}
\usepackage{esint}

\usepackage[a4paper,margin=3cm]{geometry}
\usepackage[citecolor=blue,colorlinks]{hyperref}

\usepackage{enumerate}
\usepackage{xcolor}
\makeatletter
\g@addto@macro\@floatboxreset\centering
\makeatother

\makeatletter
\def\newaliasedtheorem#1[#2]#3{
  \newaliascnt{#1@alt}{#2}
  \newtheorem{#1}[#1@alt]{#3}
  \expandafter\newcommand\csname #1@altname\endcsname{#3}
}
\makeatother

\numberwithin{equation}{section}

\newtheoremstyle{slanted}{\topsep}{\topsep}{\slshape}{}{\bfseries}{.}{.5em}{}

\theoremstyle{plain}
\newtheorem{theorem}{Theorem}[section]
\newaliasedtheorem{proposition}[theorem]{Proposition}
\newaliasedtheorem{lemma}[theorem]{Lemma}
\newaliasedtheorem{corollary}[theorem]{Corollary}
\newaliasedtheorem{counterexample}[theorem]{Counterexample}

\theoremstyle{definition}
\newaliasedtheorem{definition}[theorem]{Definition}
\newaliasedtheorem{question}[theorem]{Question}
\newaliasedtheorem{example}[theorem]{Example}
\newaliasedtheorem{conjecture}[theorem]{Conjecture}

\theoremstyle{remark}
\newaliasedtheorem{remark}[theorem]{Remark}
\newaliasedtheorem{comment}[theorem]{Comment}

\let\phi\varphi

\newcommand{\di}{\mathop{}\!\mathrm{d}}

\newcommand{\Ch}{{\sf Ch}}
\DeclareMathOperator{\Lip}{Lip}

\newcommand{\dist}{\mathsf{d}}

\newcommand{\meas}{\mathfrak{m}}

\DeclareMathOperator{\CD}{CD}
\DeclareMathOperator{\RCD}{RCD}

\newfont{\tmpf}{cmsy10 scaled 2500}

\begin{document}
\title{Isometric immersions of $\RCD$ spaces}
\author{
Shouhei Honda
\thanks{Tohoku University, \url{shouhei.honda.e4@tohoku.ac.jp}}
 } \maketitle

\begin{abstract} 
We prove that if an RCD space has a regular isometric immersion in a Euclidean space, then the immersion is a locally bi-Lipschitz embedding map.
This result leads us to prove that if a compact non-collapsed RCD space has an isometric immersion in a Euclidean space via an eigenmap, then the eigenmap is a locally bi-Lipschitz embedding map to a sphere, which generalizes a fundamental theorem of Takahashi in submanifold theory to a non-smooth setting. Applications of these results include a topological sphere theorem and topological finiteness theorems, which are new even for closed Riemannian manifolds.
\end{abstract}

\tableofcontents
\section{Introduction}
Nash's smooth embedding theorem states that any finite dimensional  (not necessary complete) Riemannian manifold $(M^n, g)$ can be embedded isometrically into a Euclidean space $\mathbb{R}^N$:
\begin{equation}
\Phi:M^n \hookrightarrow \mathbb{R}^N
\end{equation}
as Riemannian manifolds, that is, the pull-back metric $\Phi^*g_{\mathbb{R}^N}$ of the flat Riemannian metric $g_{\mathbb{R}^N}$ of the Euclidean space coincides with the original Riemannian metric $g$: $\Phi^*g_{\mathbb{R}^N}=g$. It is natural to ask when such an embedding map can be chosen by a canonical one. One of the main purposes of the paper is to give an answer to this question in a non-smooth setting by adopting \textit{eigenmaps} as canonical ones, with giving applications to the smooth setting. 
\subsection{Almost isometric embedding and isometric imbedding via eigenmap}
Let $(M^n, g)$ be an $n$-dimensional closed, that is, compact without boundary, Riemannian manifold. A map $\Phi=(\phi_1, \ldots, \phi_k):M^n \to \mathbb{R}^k$ is said to be a \textit{$k$-dimensional eigenmap} if each $\phi_i$ satisfies $\Delta \phi_i+\mu_i\phi_i\equiv 0$ for some $\mu_i \in \mathbb{R}$, that is, it is an eigenfunction of $-\Delta$ on $(M^n, g)$, or vanishes everywhere. Moreover we say that $\Phi$ is \textit{irreducible} if each $\phi_i$ is not a constant function.

B\'erard-Besson-Gallot proved in \cite{BerardBessonGallot} that for any $t \in (0, \infty)$, a map $\Phi_{t, \infty}:M^n \to \ell^2$ defined by
\begin{equation}
\Phi_{t, \infty}(x):=\left(c(n)t^{(n+2)/2}e^{-\lambda_it}\psi_i(x)\right)_i
\end{equation}
is a smooth embedding map with the following asymptotic formula:
\begin{equation}\label{asympp}
\Phi_{t, \infty}^*g_{\ell^2} =  g - \frac{2t}{3}\left(\mathrm{Ric}_{M^n}^g - \frac{1}{2}\mathrm{Scal}^g_{M^n}\,  g\right) + O(t^2), \quad t \to 0^+,
\end{equation}
where $c(n)$ is a positive constant depending only on $n$ defined by
\begin{equation}
c(n):=(4\pi)^n \left(\int_{\mathbb{R}^n}|\partial_{x_1}(e^{-|x|^2/4})|^2\di x\right)^{-1},
\end{equation}
$g_{\ell^2}$ is the standard flat Riemannian metric of $\ell^2$, $\lambda_i$ denotes the $i$-th eigenvalue ($\lambda_0=0$) of $-\Delta$ counted with multiplicities, and $\psi_i$ is a corresponding eigenfunction with the standard normalization:
\begin{equation}
\frac{1}{\mathrm{vol}_gM^n}\int_{M^n}|\psi_i|^2\di \mathrm{vol}_g=1.
\end{equation}
Thus (\ref{asympp}) says that $\Phi_{t, \infty}$ is an \textit{almost} isometric embedding when $t$ is small.

Portegies established a quantitative finite dimensional reduction of this result in \cite{Portegies16}, that is, he proved: for all $\epsilon, \tau, d \in (0, \infty)$ and any $K \in \mathbb{R}$ there exists $t_0:=t_0(n, K, \epsilon, \tau, d) \in (0, 1)$ such that for any $t \in (0, t_0)$ there exists $N_0:=N_0(n, K, \epsilon, \tau, d, t) \in \mathbb{N}$ such that if $(M^n, g)$ satisfies $\mathrm{diam}(M^n, \dist_g) \le d, \mathrm{Ric}_{M^n}^g \ge K$ and $\mathrm{inj}_{M^n}^g\ge \tau$, where $\mathrm{diam}(M^n, \dist_g)$ and $\mathrm{inj}_{M^n}^g$ denote the diameter and the injectivity radius, respectively, then for any $m \in \mathbb{N}_{\ge N_0}$, an irreducible eigenmap $\Phi_{t, m}:M^n \to \mathbb{R}^{m}$ defined by 
\begin{equation}\label{008877}
\Phi_{t, m}(x):=\left(c(n)t^{(n+2)/2}e^{-\lambda_it}\psi_i(x)\right)_{i=1}^m
\end{equation}
is a smooth embedding map with
\begin{equation}\label{78}
\|\Phi_{t, m}^*g_{\mathbb{R}^{m}}-g\|_{L^{\infty}}<\epsilon.
\end{equation}
Thus (\ref{78}) says that for any closed Riemannian manifold, the Riemannian metric can be approximated by the pull-back metric of an eigenmap. Let us emphasize that for a fixed $i \in \mathbb{N}$, we have $\|c(n)t^{(n+2)/2}e^{-\lambda_it}\psi_i\|_{L^{\infty}} \to 0$ as $t \to 0^+$.

Next we discuss the equality case, that is, let us assume that there exists an eigenmap $\Phi:M^n \to \mathbb{R}^{k+1}$ (which is not necessary an embedding map) with
\begin{equation}\label{tttsxxxx}
\Phi^*g_{\mathbb{R}^{k+1}}=g.
\end{equation}

A fundamental theorem in submanifold theory proved in \cite{takahashi} by Takahashi states that (\ref{tttsxxxx}) implies that $\Phi$ is a minimal immersion in a sphere $\mathbb{S}^k(a):=\{x \in \mathbb{R}^{k+1}; |x|_{\mathbb{R}^{k+1}}=a\}$ for some $a \in (0, \infty)$. Conversely any minimal immersion of $(M^n, g)$ in $\mathbb{S}^k(a)$ can be obtained by an eigenmap along this way. Next let us discuss the above for non-smooth spaces, so-called $\RCD(K, N)$ spaces whose introduction is given in the next subsection.
\subsection{Non-smooth spaces: $\RCD(K, N)$ spaces}

A triple $(X, \dist, \meas)$ is said to be a \textit{metric measure space} if $(X, \dist)$ is a complete separable metric space and $\meas$ is a Borel measure on $X$ with full support. The \textit{curvature-dimension condition $\CD(K, N)$} introduced in \cite{LottVillani} by Lott-Villani and \cite{Sturm06a, Sturm06b} by Sturm, independently, for metric measure space $(X, \dist, \meas)$ defines the lower bound $K \in \mathbb{R}$ on the Ricci curvature and the upper bound $N \in [1, \infty]$ of the dimension of $(X, \dist, \meas)$ in a synthetic sense. Then adding a Riemannian structure to a metric measure space satisfying the $\CD(K, N)$ condition, \textit{$\RCD(K, N)$ spaces} are introduced in \cite{AmbrosioGigliSavare14} by Ambrosio-Gigli-Savar\'e (when $N=\infty$), \cite{AmbrosioMondinoSavare} by Ambrosio-Mondino-Savar\'e (treating $\RCD^*(K, N)$ spaces), \cite{Gigli13, Gigli1} by Gigli (treating the infinitesimally Hilbertian condition), and \cite{ErbarKuwadaSturm} by Erbar-Kuwada-Sturm (treating $\RCD^*(K, N)$ spaces). Roughly speaking, $(X, \dist, \meas)$ is said to be an $\RCD(K, N)$ space if the following holds:
\begin{itemize}
\item $\mathrm{Ric}_{(X, \dist, \meas)}\ge K$, $\dim_{(X, \dist, \meas)} \le N$, and the Sobolev space $H^{1, 2}(X, \dist, \meas)$ is a Hilbert space.  
\end{itemize}
See subsection \ref{rcddef1} for the precise definition.

Thanks to quick developments on the study of $\RCD(K, N)$ spaces, we have nice structure results on $\RCD(K, N)$ spaces. For example it is proved in \cite{BrueSemola} by Bru\`e-Semola that for any $\RCD(K, N)$ space $(X, \dist, \meas)$ with $N<\infty$, the \textit{essential dimension}, denoted by $\dim_{\dist, \meas}(X) \in \mathbb{N} \cap [1, N]$, is well-defined by the unique $n \in \mathbb{N}$ satisfying that the $n$-dimensional regular set has positive $\meas$-measure. This gave a generalization of a result proved in \cite{ColdingNaber} by Colding-Naber to $\RCD(K, N)$ spaces.
 
Typical examples of $\RCD(K, N)$ spaces can be found in weighted Riemannian manifolds $(M^n, \dist_g, e^{-f}\mathrm{vol}_g)$, where $f \in C^{\infty}(M^n)$. That is, $(M^n, \dist_g, e^{-f}\mathrm{vol}_g)$ is an $\RCD(K, N)$ space if and only if $N \ge n$ holds, and the $N$-Bakry-\'Emery Ricci curvature (the (LHS) of (\ref{joju})) is bounded from below by $K$:
\begin{equation}\label{joju}
\mathrm{Ric}_{M^n}^g+\mathrm{Hess}_f^g-\frac{\dist f\otimes \dist f}{N-n}\ge Kg.
\end{equation}

Let us fix a compact $\RCD(K, N)$ space $(X, \dist, \meas)$ with $N<\infty$, that is, it is an $\RCD(K, N)$ space, and $(X, \dist)$ is a compact metric space. 
Then as in the case of closed Riemannian manifolds, the spectrum of $-\Delta$ on $(X, \dist, \meas)$ is discrete and unbounded. In particular an eigenmap $\Phi:X \to \mathbb{R}^{k}$ on $(X, \dist, \meas)$ makes sense. Note that for a weighted Riemannian manifold $(M^n, \dist_g, e^{-f}\mathrm{vol}_g)$, the corresponding Laplacian $\Delta_f$ as an $\RCD(K, N)$ space coincides with the Witten Laplacian: $\Delta_fh=\mathrm{tr}(\mathrm{Hess}_h)-g(\nabla f, \nabla h)$.

Then it is natural to ask whether the Riemannian metric $g_X$, which is also well-defined (Proposition \ref{Riemdef}), can be approximated by the pull-back metric of an eigenmap or not. A positive answer to this question was obtained in \cite{AHPT} by Ambrosio-Portegies-Tewodrose and the author for \textit{non-collapsed} $\RCD(K, n)$ spaces whose definition is introduced in \cite{DePhillippisGigli} by De Philippis-Gigli as a synthetic counterpart of non-collapsed Ricci limit spaces  (Definition \ref{noncodef}). Note that non-collapsed $\RCD(K, n)$ spaces have nicer properties rather than that of general $\RCD(K, N)$ spaces (Theorem \ref{thm:bishop}). 

A main result of \cite{AHPT} is stated as follows: for any $p \in[1, \infty)$ and all $\epsilon, \tau, d \in (0, \infty)$ and any $K \in \mathbb{R}$ there exists $t_0=t_0(n, K, \epsilon, \tau, d, p) \in (0, 1)$ such that for any $t \in (0, t_0)$ there exists $N_0:=N_0(n, K, \epsilon, \tau, d, p, t) \in \mathbb{N}$ such that if a non-collapsed $\RCD(K, n)$ space $(X, \dist, \mathcal{H}^n)$, where $\mathcal{H}^n$ denotes the $n$-dimensional Hausdorff measure,  satisfies $\mathrm{diam}(X, \dist) \le d$ and $\mathcal{H}^n(X) \ge \tau$, then for any $m \in \mathbb{N}_{\ge N_0}$, an eigenmap $\Phi_{t, m}:X \to \mathbb{R}^{m}$ defined by (\ref{008877}) with the standard normalization of eigenfunctions $\psi_i$ similarly
satisfies
\begin{equation}\label{781}
\|\Phi_{t, m}^*g_{\mathbb{R}^{m}}-g_X\|_{L^{p}}<\epsilon.
\end{equation}
Let us remark that we do not know whether $\Phi_{t, m}$ is actually a topological embedding or not.
It is also worth pointing out that this was new even for closed Riemannian manifolds and that the finiteness of $p$, $p<\infty$, is sharp which is a different point from the case of closed Riemannian manifolds. For example, any $n$-dimensional closed ball in $\mathbb{R}^n$ with $\mathcal{H}^n$ is a non-collapsed $\RCD(0, n)$ space, but it does not satisfy the $L^{\infty}$-version of (\ref{781}). Compare with (\ref{78}). Therefore (\ref{781}) says that for any compact non-collapsed $\RCD(K, n)$ space, the Riemannian metric can be approximated by the pull-back metric of an eigenmap in the $L^p$-sense for any $p \in [1, \infty)$.
\subsection{Main results}
A map $\Phi=(\phi_1, \ldots, \phi_k):U \to \mathbb{R}^k$ from an open subset $U$ of an $\RCD(K, N)$ space $(X, \dist, \meas)$ to $\mathbb{R}^k$ is said to be an \textit{isometric immersion} if it is locally Lipschitz and
\begin{equation}\label{yuyu}
\Phi^*g_{\mathbb{R}^k}:=\sum_{i=1}^k\dist \phi_i\otimes \dist \phi_i=g_X
\end{equation}
holds as $L^{\infty}_{\mathrm{loc}}$-tensors. We say that $\Phi$ is \textit{regular} if each $\phi_i$ is included in the domain of the (local) Laplacian with $\Delta \phi_i \in L^{\infty}(U, \meas)$. Thanks to regularity results proved in \cite{AMS} by Ambrosio-Mondino-Savar\'e and \cite{Jiang} by Jiang, independently, any regular map is locally Lipschitz.
Our first main result is stated as follows.
\begin{theorem}[Isometric immersion implies locally bi-Lipschitz embedding]\label{thm:bilip}
Let $(X, \dist, \meas)$ be an $\RCD(K, N)$ space with $N<\infty$ and $n=\dim_{\dist, \meas}(X)$, let $U$ be an open subset of $X$ and let $\Phi:U \to \mathbb{R}^k$ be a regular isometric immersion. Then $\Phi$ is a locally bi-Lipschitz embedding. Moreover we have the following:
\begin{enumerate}
\item For any $x \in U$ and any $\epsilon \in (0, 1)$ there exists $r \in (0, 1)$ such that $\Phi|_{B_r(x)}$ is a bi-Lipschitz embedding and that $\Phi|_{B_r(x)}$ and $(\Phi|_{B_r(x)})^{-1}$ are $(1 +\epsilon)$-Lipschitz.
\item $U$ is locally Reifenberg flat, that is, for any $x \in U$ and any $\epsilon \in (0, 1)$ there exists $r_0 \in (0, 1)$ such that 
\begin{equation}
\dist_{\mathrm{GH}}(B_r(y), B_r(0_n))<\epsilon r,\quad \forall y\in B_{r_0}(x),\,\forall r\in (0, r_0)
\end{equation}
holds, where $\dist_{\mathrm{GH}}$ denotes the Gromov-Hausdorff distance.
\end{enumerate}
In particular $U$ is homeomorphic to an $n$-dimensional topological manifold without boundary.
\end{theorem}
It is worth pointing out that the regularity assumption of $\Phi$ in Theorem \ref{thm:bilip} is essential. In order to check this, let us give two simple examples.
The first one is to consider a map $\Phi:\mathbb{R} \to \mathbb{R}$ defined by $\Phi(x)=|x|$ which is globally Lipschitz with $\Phi^*g_{\mathbb{R}}=g_{\mathbb{R}}$ as $L^{\infty}$-tensors. Thus $\Phi$ is an isometric immersion. However $\Phi$ is not locally bi-Lipschitz around the origin. Note that $\Phi$ is not regular, in fact, $\Phi$ is not included in the domain of the Laplacian. Compare with \cite{ColdingNaber2} on embeddings of non-collapsed Ricci limit spaces via (truncated) distance functions.

The second one is to consider a canonical inclusion map $\iota=(x_1, \ldots, x_n)$ from the closed unit ball $\overline{B}_1(0_n):=\{x \in \mathbb{R}^n;|x|_{\mathbb{R}^n}=1\}$ to $\mathbb{R}^n$ (recall that $(\overline{B}_1(0_n), \dist_{\mathbb{R}^n}, \mathcal{H}^n)$ is a compact non-collapsed $\RCD(0, n)$ space and that it is not Reifenberg flat in the sense above because the boundary coincides with the singular set) is an isometric immersion.
However the map is not regular because the corresponding Laplacian of $(\overline{B}_1(0_n), \dist_{\overline{B}_1(0_n)}, \mathcal{H}^n)$ as an $\RCD(0, n)$ space coincides with the Neumann Laplacian. Since each coodinate function $x_i$ does not satisfy the Neumann boundary condition, $x_i$ is not included in the domain of the Laplacian. More precise description of these observations can be found in Remark \ref{remarkremark}.

Our second main result is stated as follows, which gives a generalization of Takahashi's theorem to $\RCD(K, N)$ spaces except for the minimality parts.\footnote{The author does not know a suitable definition of \textit{minimal immersions} from $\RCD(K, N)$ spaces to Riemannian manifolds.}
Recall again that thanks to regularity results obtained in \cite{AMS}, \cite{Jiang} and in \cite{JiangLiZhang} by Jiang-Li-Zhang, any eigenmap is a regular Lipschitz map.
\begin{theorem}[Isometric immersion via eigenmap]\label{mthm4}
Let $(X, \dist, \meas)$ be a compact $\RCD(K, N)$ space with $N<\infty$ and $n=\dim_{\dist, \meas}(X)$ and let $\Phi:X \to \mathbb{R}^{k+1}$ be an eigenmap. Assume that $\Phi$ is an isometric immersion.
Then the following two conditions are equivalent:
\begin{enumerate}
\item[(a)] $|\Phi|$ is a constant function.
\item[(b)] $(X, \dist, \mathcal{H}^n)$ is a non-collapsed $\RCD(K, n)$ space with 
\begin{equation}
\meas= \frac{\meas (X)}{\mathcal{H}^n(X)}\mathcal{H}^n.
\end{equation}
\end{enumerate}
Furthermore if (a) (or (b), equivalently) holds, then the following holds.
\begin{enumerate}
\item $\Phi:(X, \dist) \to (\mathbb{S}^k(|\Phi|), \dist_{\mathbb{S}^k(|\Phi|)})$ is $1$-Lipschitz.
\item $\Phi:(X, \dist) \to (\mathbb{S}^k(|\Phi|), \dist_{\mathbb{S}^k(|\Phi|)})$ is a locally bi-Lipschitz embedding map. More precisely for any $x \in X$ and any $\epsilon \in (0, 1)$ there exists $r>0$ such that $\Phi|_{B_r(x)}$ is a bi-Lipschitz embedding and that $(\Phi|_{B_r(x)})^{-1}$ is $(1+\epsilon)$-Lipschitz.
\item If $k=n$, then $\Phi$ gives an isometry from $(X, \dist)$ to $(\mathbb{S}^n(|\Phi|), \dist_{\mathbb{S}^n(|\Phi|)})$.
\end{enumerate}
\end{theorem}
Note that it is a direct consequence of Theorem \ref{thm:bilip} that $k \ge n$ holds in Theorem \ref{mthm4} because Theorem \ref{thm:bilip} yields that $X$ is homeomorphic to an $n$-dimensional closed topological manifold, and there is no local homeomorphism from an $n$-dimensional closed topological manifold $M^n$ into $\mathbb{R}^l$ for $l \le n$. It is worth pointing out that this observation leads us to get gap theorems (Theorems \ref{mthm5} and \ref{mthm5900000}).

Let us introduce applications of Theorem \ref{mthm4}. Before that, let us define a non-negative number $L(\Phi) \in [0, \infty)$ of an eigenmap $\Phi=(\phi_1, \ldots, \phi_{k+1})$ from a compact $\RCD(K, N)$ space $(X, \dist, \meas)$ to $\mathbb{R}^{k+1}$ by 
\begin{equation}
L(\Phi):=\min_{1 \le i \le k+1}\frac{1}{\meas (X)}\int_X|\phi_i|^2\di \meas.
\end{equation}

We will also establish convergence/compactness results for eigenmaps with respect to the measured Gromov-Hausdorff convergence (Proposition \ref{ssx} and Corollary \ref{ssx2}).
Since compactness results for sequences of $\RCD(K, N)$ spaces are already known (Theorems \ref{wnon} and \ref{hohohonoho}), combining these compactness results for eigenmaps and for sequences of $\RCD(K, N)$ spaces with (5) of Theorem \ref{mthm4} gives us the following result which is new even for closed Riemannian manifolds.
\begin{theorem}[Almost characterization of sphere via eigenmap]\label{mthm2}
We have the following:
\begin{enumerate}
\item For any $K \in \mathbb{R}$, any $n \in \mathbb{N}$ and all $\epsilon, \tau, d \in (0, \infty)$ there exists $\delta:=\delta(n, K, \epsilon, \tau, d) \in (0, 1)$ such that if a compact non-collapsed $\RCD(K, n)$ space $(X, \dist, \mathcal{H}^n)$ satisfies $\mathrm{diam} (X, \dist)\le d$ and
\begin{equation}
\frac{1}{\mathcal{H}^n(X)}\int_X|\Phi^*g_{\mathbb{R}^{n+1}}-g_X|\di \mathcal{H}^n<\delta
\end{equation}
for some irreducible $(n+1)$-dimensional eigenmap $\Phi:X \to \mathbb{R}^{n+1}$ with $L (\Phi) \ge \tau$, then $\dist_{\mathrm{GH}}(X, \mathbb{S}^n(a))<\epsilon$ holds, where 
\begin{equation}
a^2=\frac{1}{\mathcal{H}^n(X)}\int_X|\Phi |_{\mathbb{R}^{n+1}}^2\di \mathcal{H}^n.
\end{equation}
\item For any $K \in \mathbb{R}$, any $n \in \mathbb{N}$ and all $\epsilon, d \in (0, \infty)$ there exists $\delta:=\delta(n, K, \epsilon, d) \in (0, 1)$ such that if a compact non-collapsed $\RCD(K, n)$ space $(X, \dist, \mathcal{H}^n)$ satisfies $\mathrm{diam} (X, \dist)\le d$ and $\dist_{\mathrm{GH}}(X, \mathbb{S}^n(a))<\delta$ for some $a \in [\tau, d/\pi]$, then there exists an irreducible $(n+1)$-dimensional eigenmap $\Phi:X \to \mathbb{R}^{n+1}$ with $|L (\Phi) - a^2(n+1)^{-1}| <\epsilon$ such that 
\begin{equation}
\frac{1}{\mathcal{H}^n(X)}\int_X|\Phi^*g_{\mathbb{R}^{n+1}}-g_X|\di \mathcal{H}^n<\epsilon
\end{equation}
holds.
\end{enumerate}
\end{theorem}
As a corollary of Theorem \ref{mthm2}, we obtain the following topological sphere theorem, which is also new even for closed Riemannian manifolds. 
\begin{theorem}[Topological sphere theorem via eigenmap]\label{thm:sphere}
For any $K \in \mathbb{R}$, any $n \in \mathbb{N}$ and all $d, \tau \in (0, \infty)$, there exists $\delta:=\delta(n, K, d, \tau) \in (0, 1)$ such that if a compact non-collapsed $\RCD(K, n)$ space $(X, \dist, \mathcal{H}^n)$ with $\mathrm{diam}(X, \dist) \le d$ satisfies that there exists an irreducible $(n+1)$-dimensional eigenmap $\Phi:X \to \mathbb{R}^{n+1}$ with $L(\Phi) \ge \tau$ such that 
\begin{equation}
\frac{1}{\mathcal{H}^n(X)}\int_X\left| \Phi^*g_{\mathbb{R}^{n+1}}-g_X\right|\di \mathcal{H}^n <\delta
\end{equation}
holds, then $X$ is homeomorphic to $\mathbb{S}^n:=\mathbb{S}^1(1)$. Moreover if $(X, \dist)$ is isometric to a Riemannian manifold, then the homeomorphism can be improved to a diffeomorphism.
\end{theorem}
Let us emphasize that in Theorems \ref{mthm2} and \ref{thm:sphere} we do not necessary assume a positive lower bound on Ricci curvature and a positive lower bound on the volume $\mathcal{H}^n(X)$, which are different from previous sphere theorems, for instance \cite{Colding1, Colding2} by Colding, \cite{KM} by Kapovitch-Mondino, \cite{HondaMondello} by Mondello and the author, and \cite{Petersen} by Petersen.

By (3) of Theorem \ref{mthm4} we know that the sphere is the only non-collapsed $\RCD(K, n)$ space whose Riemannian metric can be realized by the pull-back metric of an $(n+1)$-dimensional eigenmap. Therefore it is natural to ask:
\begin{itemize}
\item[$(\diamondsuit)$]For fixed $k$ how many non-collapsed $\RCD(K, n)$ spaces $(X, \dist, \mathcal{H}^n)$ satisfy (\ref{yuyu}) for an eigenmap $\Phi$?
\end{itemize}
In order to give an answer to this question, we will provide two topological finiteness theorems in more general setting. Roughly speaking, they tell us that from the point of view of topology, only finite spaces realize this condition. See Theorems \ref{topological1} and \ref{topological2}. 

Let us explain how to acheive these results in the next subsection.
\subsection{Outline of the proofs}
Let us start by a simple observation.
A (globally) Lipschitz map $\Phi:=(\phi_1,\ldots, \phi_k):\mathbb{R}^n \to \mathbb{R}^k$ is an isometric embedding as metric spaces (that is, $\Phi$ preserves the distance) if and only if each $\phi_i$ is a harmonic function on $\mathbb{R}^n$ with
\begin{equation}
\Phi^*g_{\mathbb{R}^k}=g_{\mathbb{R}^n}.
\end{equation}
The harmonicity of $\phi_i$ is essential in this observation because, for example, as already observed, the map $\Phi:\mathbb{R} \to \mathbb{R}$ defined by $\Phi(x) =|x|$ is globally Lipschitz and satisfies $\Phi^*g_{\mathbb{R}}=g_{\mathbb{R}}$ as $L^{\infty}$-tensors, however it is not an isometric embedding as metric spaces. 

The first step to prove Theorem \ref{thm:bilip} is to establish a quantitative version of this observation, that is, roughly speaking, 
if $\Phi^*g_{\mathbb{R}^k}$ is close to $g_X$ around a point $x$ of an $\RCD(K, N)$ space $(X, \dist, \meas)$, then $\Phi$ gives a Gromov-Hausdorff approximation to the image on a neighbourhood of $x$, and the point $x$ is almost $n$-dimensional regular, where $n=\dim_{\dist, \meas}(X)$ (Theorem \ref{highregular}). This is done by a blow-up argument with (\ref{yuyu}) and stability results of Sobolev functions with respect to the measured Gromov-Hausdorff convergence established in \cite{AmbrosioHonda, AmbrosioHonda2} by Ambrosio and the author. Let us emphasize that the regularity assumption of $\Phi$ is essentially used here.
The second step is to prove, by using the first step, that a regular almost isometric immersion implies a local bi-Lipschitz embedding and an almost Reifenberg flatness of the space (Theorem \ref{yhug}), which leads us to get Theorem \ref{thm:bilip}. 
 
For Theorem \ref{mthm4}, we will split the proof into several lemmas in Section \ref{secproof}.
In order to show the equivalence between (a) and (b) in Theorem \ref{mthm4}, we prove the following formula for any eigenmap $\Phi:X\to \mathbb{R}^{k+1}$ (Theorem \ref{cor:1}):
\begin{equation}\label{amann}
\nabla^*\left( \Phi^*g_{\mathbb{R}^{k+1}}\right)=-\frac{1}{4}\dist \Delta |\Phi|^2_{\mathbb{R}^{k+1}}.
\end{equation}
By using (\ref{amann}), we know that under assuming (\ref{yuyu}), $|\Phi|$ is a constant function if and only if $\nabla^*g_X=0$ holds. Moreover thanks to a development on the second order differential calculus in \cite{Gigli} by Gigli, $\nabla^*g_X=0$ holds if and only if 
\begin{equation}\label{nng}
\Delta f=\mathrm{tr}(\mathrm{Hess}_f),\quad \forall f \in D(\Delta)
\end{equation}
holds. 

Let us assume that (a) holds. Thus we have (\ref{nng}).
Then we can apply a result proved in \cite{BrueSemola}, which confirmed a conjecture raised in \cite{DePhillippisGigli}, to show that $(X, \dist, \meas)$ is a weakly non-collapsed $\RCD(K, n)$ space. 
Since any compact weakly non-collapsed $\RCD(K, n)$ space is actually a non-collapsed $\RCD(K, n)$ space up to scaling the measure by a positive constant, which is proved in \cite{Honda19} by the author, we have (b).
The converse implication, from (b) to (a), is justified along the same line with a result proved in \cite{Han} by Han (Lemma \ref{houuy}).

Under assuming (a), Theorem \ref{thm:bilip} allows us to prove (1) and (2) of Theorem \ref{mthm4} (Lemmas \ref{hoho} and \ref{uuuuu}). Moreover we know that if $k=n$, then $\Phi$ is a local isometry (Lemma \ref{esse}). In particular $(X, \dist)$ is isometric to a closed Riemannian manifold whose sectional curvature is equal to $1$. Since we can check that the first positive eigenvalue $\lambda_1$ of $(X, \dist, \meas)$ is at most $n/|\Phi|^2$, the final statement, (3) of Theorem \ref{mthm4}, follows from Obata's theorem \cite{Obata} which states that if a closed $n$-dimensional Riemannian manifold $(M^n, g)$ satisfies $\mathrm{Ric}_{M^n}^g \ge n-1$ and $\lambda_1\le n$, then $(M^n, g)$ is isometric to $(\mathbb{S}^n, g_{\mathbb{S}^n})$ (Lemma \ref{complete}).  

The remaining results, Theorems \ref{mthm2} and \ref{thm:sphere}, are justified by contradiction after combining Theorem \ref{mthm4} with compactness results for eigenmaps established in Section \ref{sseewwe}. Roughly speaking, we will prove the following in Section \ref{sseewwe}:
\begin{itemize}
\item The set of all compact $\RCD(K, N)$ spaces $(X, \dist, \meas, \Phi)$ with irreducible eigenmaps $\Phi:X \to \mathbb{R}^{k}$ satisfying that $\mathrm{diam}(X, \dist)\le d$ and $L(\Phi)\ge \tau$, is compact with respect to the joint convergences of measured Gromov-Hausdorff convergence of $(X, \dist, \meas)$ and of irreducible eigenmaps $\Phi$. 
\end{itemize}
See Theorem \ref{ssx} and Corollary \ref{ssx2}.

\subsection{Organization}
In the next section we will introduce knowledges on the $\RCD$ theory. In Section \ref{immersion} we will discuss isometric immersions of $\RCD(K, N)$ spaces, in particular, we will prove Theorem \ref{thm:bilip}. We will study eigenmaps on $\RCD(K, N)$ spaces in Section \ref{sseewwe}. Section \ref{secproof} is devoted to the proof of Theorem \ref{mthm4}. Finally, we will prove remaining statements introduced above in Sections \ref{6666r6r6r6} and \ref{finiteness}.

\smallskip\noindent
\textbf{Acknowledgement.}
The author would like to thank Shin Nayatani and Toshihiro Shoda for valuable suggestions. He is grateful to the referee for reading the paper carefully and for giving many valuable suggestions.
He acknowledges the supports of the Grant-in-Aid for Scientific Research (B) of 20H01799 and the Grant-in-Aid for Scientific Research (B) of 18H01118.
\section{Preliminaries}
In this section we give a quick introduction to the $\RCD$ theory in order to understand the paper under assuming a bit of knowledges of metric measure geometry. Therefore we will sometimes refer only to suitable references for the details. For example we omit the definitions of \textit{Gromov-Hausdorff (GH) convergence}, \textit{measured Gromov-Hausdorff (mGH) convergence}, and \textit{pointed measured Gromov-Hausdorff (pmGH) convergence} which are metrizable by $\dist_{\mathrm{GH}}, \dist_{\mathrm{mGH}}$, and $\dist_{\mathrm{pmGH}}$\footnote{More precisely, we should use $\mathbf{D}, p\mathbb{G}_w$ instead of using $\dist_{\mathrm{mGH}}, \dist_{\mathrm{pmGH}}$ due to \cite{GigliMondinoSavare13, Sturm06a}. However in order to keep our notation simply we will use these notations.}, respectively.  See \cite{BBI, GigliMondinoSavare13, LottVillani, Sturm06a, Sturm06b, Villani}.
\subsection{$\RCD$ space}\label{rcddef1}
We refer to \cite{A, AmbrosioGigliMondinoRajala, AmbrosioGigliSavare13, AmbrosioGigliSavare14, AmbrosioMondinoSavare, CM, ErbarKuwadaSturm, LottVillani, Sturm06a, Sturm06b} as references in this subsection.

We recall the definition of metric measure spaces again: a triple $(X, \dist, \meas)$ is said to be a \textit{metric measure space} if $(X, \dist)$ is a complete separable metric space and $\meas$ is a Borel measure on $X$ with full support. We fix a metric measure space $(X, \dist, \meas)$ below.

Define the \textit{Cheeger energy} $\Ch:L^2(X,\meas)\to [0,\infty]$ by
\begin{equation}\label{eq:defchee}
\Ch(f):=\inf\left\{\liminf_{i\to\infty}\int_X{\rm lip}^2f_i\di\meas:\ f_i\in\Lip (X,\dist)\cap (L^2 \cap L^{\infty})(X, \meas),\,\,\,\|f_i-f\|_{L^2}\to 0
\right\},
\end{equation}
where
${\rm lip}f(x)$
denotes the local slope of $f$ at $x$, that is, 
\begin{equation}
{\rm lip}f(x):=\lim_{r\to 0^+}\sup_{y \in B_r(x) \setminus \{x\}}\frac{|f(x)-f(y)|}{\dist (x, y)}
\end{equation}
if $x$ is not isolated, ${\rm lip}f(x):=0$ otherwise,  and $\Lip (X,\dist)$ is the set of all Lipschitz functions on $X$.
Then the \textit{Sobolev space} $H^{1, 2}=H^{1,2}(X,\dist,\meas)$ is defined as the finiteness domain of $\Ch$ in $L^2(X, \meas)$ and it is a Banach space equipped with the norm $\|f\|_{H^{1, 2}}=\sqrt{\|f\|_{L^2}^2+\Ch(f)}$.
We are now in a position to introduce the definition of $\RCD(K, N)$ space. More precisely the following definition should be refered as \textit{$\RCD^*(K, N)$ spaces} treated in \cite{AmbrosioMondinoSavare, ErbarKuwadaSturm}, however the equivalence between $\RCD(K, N)$ and $\RCD^*(K, N)$ spaces is established in \cite{CM} under the finite measure assumption. Since we will work mainly for compact spaces, we use this simplified notation.
\begin{definition}[$\RCD(K, N)$ space]
$(X, \dist, \meas)$ is said to be an $\RCD(K, N)$ \textit{space} for some $K \in \mathbb{R}$ and some $N \in [1, \infty]$ if the following four conditions hold:
\begin{itemize}
\item{(Volume growth condition)} There exist $C\in (0, \infty)$ and $x \in X$ such that $\meas (B_r(x)) \le Ce^{Cr^2}$ holds for any $r \in (0, \infty)$.
\item{(Infinitesimally Hilbertian condition)} $H^{1, 2}$ is a Hilbert space. In particular for all $f_i \in H^{1, 2} (i=1, 2)$,
\begin{equation}
\langle \nabla f_1, \nabla f_2\rangle:=\lim_{t \to 0}\frac{|\nabla (f_1+tf_2)|^2-|\nabla f_1|^2}{2t} \in L^1(X, \meas)
\end{equation}
is well-defined, where $|\nabla f_i|$ denotes the minimal relaxed slope of $f_i$.
\item{(Sobolev-to-Lipschitz property)} Any function $f \in H^{1, 2}$ satisfying $|\nabla f|(y) \le 1$ for $\meas$-a.e. $y \in X$ has a $1$-Lipschitz representative.  
\item{(Bochner inequality)} For any $f \in D(\Delta)$ with $\Delta f \in H^{1, 2}$, we have
\begin{equation}
\frac{1}{2}\int_X\Delta \phi |\nabla f|^2\di \meas \ge \int_X \phi \left(\frac{(\Delta f)^2}{N}+\langle \nabla \Delta f, \nabla f\rangle +K|\nabla f|^2\right)\di \meas
\end{equation}
for any $\phi \in D(\Delta) \cap L^{\infty}(X, \meas)$ with $\Delta \phi \in L^{\infty}(X, \meas)$ and $\phi \ge 0$, where
\begin{equation}
D(\Delta):=\left\{ f \in H^{1, 2}; \exists h=:\Delta f \in L^2, \,\mathrm{s.t.} \, \int_X\langle \nabla f, \nabla \psi \rangle \di \meas =-\int_Xh\psi \di \meas,\,\forall \psi \in H^{1, 2}\right\}.
\end{equation}
\end{itemize}
\end{definition}
Throughout the paper the parameters $K \in \mathbb{R}$ and $N \in [1, \infty)$ will be kept fixed. 
\subsection{Structure of $\RCD(K, N)$ space}
Let $(X, \dist, \meas)$ be an $\RCD(K, N)$ space with $\mathrm{diam}(X, \dist)>0$. It is known that $(X, \dist)$ is a proper geodesic space (see \cite[Cor.1.4]{GRS} for more information on geodesics).
\begin{definition}[Regular set $\mathcal{R}_k$]
For any $k \geq 1$, we denote by $\mathcal{R}_k$ the \textit{$k$-dimensional regular set}  of $(X, \dist, \meas)$, 
namely, the set of all points $x \in X$ such that $(X, r^{-1}\dist, \meas (B_{r}(x))^{-1}\meas, x)$ pmGH converge to $(\mathbb{R}^k, \dist_{\mathbb{R}^k}, \omega_k^{-1}\mathcal{H}^k,0_k)$ as $r \to 0^+$, where $\omega_k:=\mathcal{H}^k(B_1(0_k))$.
\end{definition}
The following result is proved in \cite[Thm.0.1]{BrueSemola} after \cite{MondinoNaber} which gives a generalization of \cite[Thm.1.12]{ColdingNaber} to $\RCD(K, N)$ spaces.
\begin{theorem}[Essential dimension of $\RCD (K,N)$ spaces]\label{thmRCD decomposition}
Let $(X,\dist,\meas)$ be an $\RCD (K,N)$ space. Then, there exists a unique integer $n\in [1,N] \cap \mathbb{N}$, denoted by $\dim_{\dist,\meas}(X)$, such that
 \begin{equation}\label{eq:regular set is full}
\meas(X\setminus \mathcal{R}_n\bigr)=0
\end{equation}
holds.
\end{theorem}
The following is a direct consequence of the Bishop-Gromov inequality \cite[Thm.5.31]{LottVillani}, \cite[Thm2.3]{Sturm06b} (see \cite[Thm.30.11]{Villani}) and the Poincar\'e inequality \cite[Thm.1]{Rajala} with \cite[Thm.5.1]{HK}.
\begin{theorem}[Rellich compactness]\label{thm:rellich}
If $(X, \dist)$ is compact (or equivalently, $\mathrm{diam}(X, \dist)<\infty$ by properness), then the canonical inclusion map:
\begin{equation}
\iota:H^{1, 2}(X, \dist, \meas) \hookrightarrow L^2(X, \meas)
\end{equation}
is a compact operator.
\end{theorem}
\begin{definition}[Eigenfunction]
A function $f \in L^2(X, \meas)$ is said to be \textit{an eigenfunction (of $-\Delta$) on $(X, \dist, \meas)$} if $f \in D(\Delta)$ holds with $f \not \equiv 0$ and $\Delta f+\lambda f\equiv 0$ for some $\lambda \in \mathbb{R}$, where $\lambda$ is called the \textit{eigenvalue of $f$}. 
\end{definition}
A direct consequence of Theorem \ref{thm:rellich} is that if $(X, \dist)$ satisfies $\mathrm{diam}(X, \dist)\le d<\infty$, then $-\Delta$ admits a discrete non-negative spectrum: 
\begin{equation}
0=\lambda_0 < \lambda_1 \le \lambda_2 \le \cdots \to \infty,
\end{equation}
where $\lambda_i:=\lambda_i(X, \dist, \meas)$ is the $i$-th eigenvalue of $-\Delta$ counted with multiplicities.
We denote the corresponding eigenfunctions by $\psi_0, \psi_1, \ldots $ with the standard normalization:
\begin{equation}
\frac{1}{\meas(X)}\int_X|\psi_i|^2\di \meas=1.
\end{equation} 
It is known that $\psi_i$ is Lipschitz, in fact, it holds that
\begin{equation}\label{eq:eigenest}
\|\psi_i\|_{L^\infty} \leq C \lambda_i^{N/4}, \qquad \| \nabla \psi_i \|_{L^\infty} \leq C \lambda_i^{(N+2)/4}, \qquad \lambda_i \ge C^{-1}i^{2/N},
\end{equation}
where $C:=C(d, K, N)>1$. See \cite{Jiang, JiangLiZhang} (see also \cite{AHPT}).

Let us recall the \textit{Sobolev space} $H^{1, 2}(U, \dist, \meas)$ for an open subset $U$ of $X$.
A function $f \in L^2(U, \meas)$ belongs to $H^{1, 2}(U, \dist, \meas)$ if and only if $\phi f\in H^{1, 2}(X, \dist, \meas)$ holds for any $\phi \in \mathrm{Lip}(X, \dist)$ with compact support in $U$, and $|\nabla f| \in L^2(U, \meas)$ is satisfied. See also \cite{Cheeger, Shanmugalingam}.
Then we denote by $D(\Delta, U)$ the set of all $f \in H^{1, 2}(U, \dist, \meas)$ satisfying that there exists a unique $h \in L^2(U, \meas)$, denoted by $\Delta_Uf$ (or $\Delta f$ for short), such that 
\begin{equation}
\int_U\langle \nabla f, \nabla \phi \rangle  \di \meas =-\int_Uh\phi \di \meas
\end{equation}
holds for any $\phi \in \Lip (X, \dist)$ with compact support in $U$.

Finally let us end this subsection by introducing (global) harmonic functions.
\begin{definition}[Harmonic function]
A function $f:X \to \mathbb{R}$ is said to be \textit{harmonic} if $f|_U \in D(\Delta, U)$ holds with $\Delta f=0$ for any bounded open subset $U$ of $X$.
\end{definition}
See also \cite{AmbrosioHonda2, AHPT}.
\subsection{Non-collapsed $\RCD$ space}
Let us recall a special class of $\RCD(K, N)$ spaces, so-called \textit{non-collapsed space}, introduced in \cite{DePhillippisGigli} in order to give a synthetic counterpart of non-collapsed Ricci limit spaces whose study was developed in \cite{CheegerColding1, CheegerColding2}. Note that a similar notion of non-collapsed space is provided in \cite{Kita}, which is a-priori weaker than that of \cite{DePhillippisGigli} (however it is conjectured in \cite{DePhillippisGigli} that they are essentially equivalent each other as discussed below).
\begin{definition}[Non-collapsed $\RCD(K, N)$ space]\label{noncodef}
An $\RCD(K, N)$ space $(X, \dist, \meas)$ is said to be \textit{non-collapsed} if $\meas=\mathcal{H}^N$ holds.
\end{definition}
Non-collapsed $\RCD(K, N)$ spaces have nicer properties rather than that of general $\RCD(K, N)$ spaces. Let us introduce some of them: 
\begin{theorem}\label{thm:bishop}
Let $(X, \dist, \mathcal{H}^N)$ be a non-collapsed $\RCD(K, N)$ space. Then the following holds.
\begin{enumerate}
\item We have $\dim_{\dist, \meas}(X)=N$.
\item It holds that for any $x \in X$,
\begin{equation}\label{ref}
\lim_{r \to 0^+}\frac{\mathcal{H}^N(B_r(x))}{\omega_Nr^N} \le 1.
\end{equation}
Moreover the equality in (\ref{ref}) is satisfied if and only if $x \in \mathcal{R}_N$ holds.
\end{enumerate}
\end{theorem}
The inequality (\ref{ref}) is sometimes refered as the \textit{Bishop inequality}. See \cite[Thm.1.3 and 1.6]{DePhillippisGigli}. 
It is worth pointing out that a quntitative version of the rigidity part of the Bishop inequality is also satisfied. In order to explain it, let us use a standard notation in the convergence theory:
\begin{equation}
\Psi (\epsilon_1, \epsilon_2, \ldots, \epsilon_l;c_1, c_2, \ldots, c_m)
\end{equation}
denotes a function $\Psi: (\mathbb{R}_{>0})^l \times \mathbb{R}^m \to (0, \infty)$ satisfying
\begin{equation}
\lim_{(\epsilon_1, \ldots, \epsilon_k) \to 0} \Psi  (\epsilon_1, \epsilon_2, \ldots, \epsilon_l;c_1, c_2, \ldots, c_m)=0, \quad \forall c_i \in \mathbb{R}.
\end{equation}
Then the quantitative rigidity of (\ref{ref}) is stated as follows: 
\begin{theorem}[Almost rigidity of Bishop inequality]\label{thm:bishop2}
Let $(X, \dist, \mathcal{H}^N)$ be a non-collapsed $\RCD(K, N)$ space and let $x \in X$.
If 
\begin{equation}\label{ttffff}
\left|\frac{\mathcal{H}^N(B_r(x))}{\omega_Nr^N} - 1\right|<\epsilon
\end{equation}
holds for some $r \in (0, 1)$, then  
\begin{equation}
\dist_{\mathrm{GH}}(B_{r/2}(x), B_{r/2}(0_N))< \Psi(\epsilon, r;K, N)r
\end{equation}
and
\begin{equation}
\dist_{\mathrm{pmGH}}\left((X, t^{-1}\dist, x, \meas (B_t(x))^{-1}\meas ), (\mathbb{R}^N, \dist_{\mathbb{R}^N}, 0_N, \omega_N^{-1}\mathcal{H}^N)\right)<\Psi (\epsilon, t/r, r; K, N), \quad \forall t \in (0, 1)
\end{equation}
hold. Conversely if
\begin{equation}
\dist_{\mathrm{GH}}(B_{r}(x), B_{r}(0_N))<\epsilon r
\end{equation}
holds for some $r \in (0, 1)$, then
\begin{equation}
\left|\frac{\mathcal{H}^N(B_r(x))}{\omega_Nr^N} - 1\right|<\Psi(\epsilon, r;K, N)
\end{equation}
is satisfied.
\end{theorem}
See \cite[Thm.1.3 and 1.6]{DePhillippisGigli} (see also \cite[Prop.6.5]{AHPT}).

In connection with (1) of Theorem \ref{thm:bishop} it is conjectured in \cite[Rem.1.11]{DePhillippisGigli} the converse implication is true up to multiplying a positive constant to the measure, that is, if a $\RCD(K, N)$ space $(X, \dist, \meas)$ satisfies $\mathrm{dim}_{\dist, \meas}(X)=N$, then $\meas=a\mathcal{H}^N$ holds for some $a \in (0 ,\infty)$. This conjecture was proved in \cite[Cor.1.3]{Honda19} when $(X, \dist)$ is compact by using $L^2$-embedding results via the heat kernel obtained in \cite{AHPT}. Namely:
\begin{theorem}\label{wnon}
Let $(X, \dist, \meas)$ be a compact $\RCD(K, N)$ space. Then $\dim_{\dist, \meas}(X)=N$ holds if and only if $\meas=a\mathcal{H}^N$ holds for some $a \in (0, \infty)$.
\end{theorem}
Finally let us end this subsection by giving the following convergence result proved in \cite[Thm.1.2]{DePhillippisGigli} (see \cite[Thm.5.9]{CheegerColding1} with \cite[Thm.0.1]{Colding} for the corresponding results on Ricci limit spaces).
\begin{theorem}[GH implies mGH]\label{GHmGH}
Let $(X_i, \dist_i, x_i, \mathcal{H}^N)$ be a sequence of pointed non-collapsed $\RCD(K, N)$ spaces. If $(X_i, \dist_i, x_i)$ pGH-converge to a pointed complete metric space $(X, \dist, x)$, then
\begin{equation}
\mathcal{H}^N(B_r(z_i)) \to \mathcal{H}^N(B_r(z))
\end{equation}
holds for any $r \in (0, \infty)$ and any $z_i \in X_i \to z \in X$.
\end{theorem}
\subsection{Second order calculus}
Let $(X, \dist, \meas)$ be an $\RCD(K, \infty)$ space.
The main purpose of this subsection is to recall the \textit{second order differential calculus} developed in \cite{Gigli}. To keep presentation short in the paper we omit several fundamental notions, for instance, 
\begin{itemize}
\item the space of all $L^2$-vector fields, $L^2$-one forms and $L^2$-tensors of type $(0, 2)$, on $A \subset X$, denoted by $L^2(T(A, \dist, \meas)), L^2(T^*(A, \dist, \meas))$ and $L^2((T^*)^{\otimes 2}(A, \dist, \meas))$, respectively;
\item the gradient operator $\nabla:H^{1, 2}(U, \dist, \meas) \to L^2(T(U, \dist, \meas))$ for an open subset $U$ of $X$, and the exterior derivative $\dist$.
\end{itemize}
We denote the pointwise Hilbert-Schmidt norm and the pointwise scalar product by $|T|_{\mathrm{HS}}$ (or $|T|$ for short) and $\langle S, T\rangle$, respectively. See \cite[Subsect.3.2]{Gigli} (see also \cite[Sect.10]{AmbrosioHonda}).
Put
\begin{equation}
\mathrm{Test}F(X, \dist, \meas):=\{f \in D(\Delta) \cap \Lip(X, \dist)\cap L^{\infty}(X, \meas); \Delta f \in H^{1, 2}(X, \dist, \meas)\}
\end{equation}
and recall that $\mathrm{Test}F(X, \dist, \meas)$ is an algebra with $|\nabla f|^2 \in H^{1, 2}(X, \dist, \meas)$ for any $f \in \mathrm{Test}F(X, \dist, \meas)$. We need the following important notion, the \textit{Hessian} of a function:
\begin{theorem}[Hessian]
For any $f \in \mathrm{Test}F(X, \dist, \meas)$ there exists a unique $T \in L^2((T^*)^{\otimes 2}(X, \dist, \meas))$, called the \textit{Hessian of f}, denoted by $\mathrm{Hess}_f$, such that for all $f_i \in \mathrm{Test}F(X, \dist, \meas)$,
\begin{equation}\label{eq:hess}
\langle T, \dist f_1\otimes \dist f_2\rangle=\frac{1}{2}\left(\langle \nabla f_1, \nabla \langle \nabla f_2, \nabla f\rangle \rangle + \langle \nabla f_2, \nabla \langle \nabla f_1, \nabla f \rangle \rangle -\langle f, \nabla \langle \nabla f_1, \nabla f_2\rangle \rangle \right)
\end{equation}
holds for $\meas$-a.e. $x \in X$. 
\end{theorem}
See \cite[Thm.3.3.8]{Gigli} and \cite[Lem.3.3]{Savare}.
Moreover it is proved in \cite[Thm.3.3.8 and Cor.3.3.9]{Gigli} and \cite[Thm.3.4]{Savare} that the Hessian is well-defined for any $f \in D(\Delta)$ by satisfying (\ref{eq:hess}), and that the Bochner inequality involving the Hessian term:
\begin{equation}\label{eq:bochner}
\frac{1}{2}\int_X|\nabla f|^2\Delta \phi \di \meas \ge \int_X\phi\left( |\mathrm{Hess}_f|^2+\langle \nabla \Delta f, \nabla f\rangle +K|\nabla f|^2\right)\di \meas
\end{equation}
holds for any $f \in \mathrm{Test}F(X, \dist, \meas)$ and $\phi \in D(\Delta)$ with $\phi \ge 0$ and $\phi, \Delta \phi \in L^{\infty}(X, \meas)$. 

Let us define the \textit{Riemannian metric} as follows. See \cite[Prop.3.2]{AHPT} and \cite[Thm.5.1]{GP} for the proof.
\begin{proposition}[Riemannian metric]\label{Riemdef}
There exists a unique $g_X \in L^{\infty}((T^*)^{\otimes 2}(X, \dist, \meas))$ such that for any $f_i \in \mathrm{Test}F(X, \dist, \meas)$ we have
\begin{equation}
\langle g_X, \dist f_1 \otimes \dist f_2\rangle =\langle \nabla f_1, \nabla f_2\rangle,\quad \text{for $\meas$-a.e. $x \in X$}. 
\end{equation}
We call $g_X$ the \textit{Riemannian metric of $(X, \dist, \meas)$}. Moreover if $(X, \dist, \meas)$ is an $\RCD(K, N)$ space with $n=\dim_{\dist, \meas}(X)$, then 
\begin{equation}
|g_X|=\sqrt{n},\quad \text{for $\meas$-a.e. $x \in X$}. 
\end{equation}
\end{proposition}
Let us introduce a relationship between $\Delta$, $g_X$ and $\mathrm{Hess}_f$.
\begin{theorem}[Laplacian is trace of Hessian under maximal dimension]\label{thm:hanresult}
Assume that $(X, \dist, \meas)$ is an $\RCD(K, N)$ space with $\mathrm{dim}_{\dist, \meas}(X)=N$. Then
for all $f \in D(\Delta)$ 
\begin{equation}\label{eq:laptrace}
\Delta f=\mathrm{tr}(\mathrm{Hess}_f)(:=\langle \mathrm{Hess}_f, g_X\rangle) \quad \text{for $\meas$-a.e. $x \in X$}.
\end{equation}
\end{theorem}
See \cite[Prop.3.2]{Han} for the proof.
\begin{definition}[Adjoint operator $\delta$]\label{def:adj}
Let us denote by $D(\delta)$ the set of all $\omega \in L^2(T^*(X, \dist, \meas))$ such that there exists a unique $f \in L^2(X, \meas)$, denoted by $\delta \omega$, such that 
\begin{equation}
\int_X\langle \omega, \dist h \rangle \di \meas= \int_Xfh\di \meas, \quad \forall h \in H^{1, 2}(X, \dist, \meas)
\end{equation}
holds. 
\end{definition}
Let us recall the space of all test $1$-forms:
\begin{equation}
\mathrm{Test}T^*(X, \dist, \meas):=\left\{\sum_{i=1}^lf_{0, i}\dist f_{1, i}; l\in \mathbb{N}, f_{j, i} \in \mathrm{Test}F(X, \dist, \meas)\right\}
\end{equation}
which is a dense subspace of $L^2(T^*(X, \dist, \meas))$.
It is proved in \cite[Prop.3.5.12]{Gigli} that $\mathrm{Test}T^*(X, \dist, \meas) \subset D(\delta)$ holds with \begin{equation}\label{eq;delta}
\delta(f_1\dist f_2)=-\langle \nabla f_1, \nabla f_2 \rangle-f_1\Delta f_2, \quad \forall f_i \in \mathrm{Test}F(X, \dist, \meas).
\end{equation}
\begin{definition}[Sobolev spaces $W^{1, 2}_C$]\label{def:cov}
Let us denote by $W^{1, 2}_C(T^*(X, \dist, \meas))$ the set of all $\omega \in L^2(T^*(X, \dist, \meas))$ such that there exists a unique $T \in L^2((T^*)^{\otimes 2}(X, \dist, \meas))$, denoted by $\nabla \omega$,  such that for all $f_i \in \mathrm{Test}F(X, \dist, \meas)(i=1, 2)$ we have
\begin{equation}
\int_X\langle T, f_0\dist f_1\otimes \dist f_2\rangle \di \meas = \int_X\left( \langle \omega, \dist f_2\rangle \delta(f_0\dist f_1)-f_0\langle \mathrm{Hess}_{f_2}, \omega \otimes \dist f_1\rangle \right) \di \meas. 
\end{equation}
\end{definition}
See \cite[Def.3.4.1]{Gigli} (note that although the definition of \cite[Def.3.4.1]{Gigli} is stated for vector fields, it is equivalent to the above under the canonical isometry: $L^2(T^*(X, \dist, \meas)) \cong L^2(T(X, \dist, \meas))$).
It is proved in \cite[Thm.3.4.2]{Gigli} that $\mathrm{Test}T^*(X, \dist, \meas) \subset W^{1, 2}_C(T^*(X, \dist, \meas))$ holds with
\begin{equation}\label{eq:leib}
\nabla (f_1\dist f_2)=\dist f_1\otimes \dist f_2 +f_1 \mathrm{Hess}_{f_2}, \quad \forall f_i \in \mathrm{Test}F(X, \dist, \meas).
\end{equation}
\begin{definition}[Adjoint operator $\nabla^*$]\label{rrtf}
Let us denote by $D(\nabla^*)$ the set of all $T \in L^2((T^*)^{\otimes 2}(X, \dist, \meas))$ such that there exists $\alpha \in L^2(T^*(X, \dist, \meas))$, denoted by $\nabla^*T$, such that 
\begin{equation}
\int_X\langle T, \nabla \omega \rangle \di \meas =\int_X\langle \alpha, \omega \rangle \di \meas, \quad \forall \eta \in \mathrm{Test}T^*(X, \dist, \meas)
\end{equation}
holds. 
\end{definition}
Note that since the space of all test forms of type $(0, 2)$:
\begin{equation}
\mathrm{Test}(T^*)^{\otimes 2}(X, \dist, \meas):=\left\{\sum_{i=1}^lf_{0, i}\dist f_{1, i} \otimes \dist f_{2, i}; l\in \mathbb{N}, f_{j, i} \in \mathrm{Test}F(X, \dist, \meas)\right\}
\end{equation}
is also dense in $L^2((T^*)^{\otimes 2}(X, \dist, \meas))$, the existence of $\alpha$ in Definition \ref{rrtf} implies the uniqueness.
\begin{proposition}\label{thm:3}
For any $\phi \in \mathrm{Test}(X, \dist, \meas)$ we have $\dist \phi \otimes \dist \phi \in D(\nabla^*)$ with 
\begin{equation}\label{eq:2}
\nabla^*(\dist \phi \otimes \dist \phi )=-\Delta \phi \dist \phi-\frac{1}{2}\dist|\dist \phi|^2.
\end{equation}
\end{proposition}
\begin{proof}
For all $f_1, f_2 \in \mathrm{Test}F(X, \dist, \meas)$ we have
\begin{align}
&\int_X\langle \dist \phi \otimes \dist \phi, \nabla (f_1\dist f_2)\rangle \di \meas \nonumber\\
&=\int_X\left( \langle \dist \phi, \dist f_1\rangle \langle \dist \phi, \dist f_2\rangle +f_1\langle \mathrm{Hess}_{f_2}, \dist \phi \otimes \dist \phi \rangle \right)\di \meas \nonumber \\
&=\int_X\left( \langle \langle \dist \phi, \dist f_2\rangle\dist \phi, \dist f_1\rangle +\frac{f_1}{2}\left\{ 2\langle \dist \phi, \dist \langle \dist f_2, \dist \phi\rangle \rangle -\langle \dist f_2, \dist |\dist \phi|^2\rangle \right\} \right)\di \meas \nonumber \\
&=\int_X\left(\delta ( \langle \dist \phi, \dist f_2\rangle\dist \phi) f_1 +f_1\langle \dist \phi, \dist \langle \dist f_2, \dist \phi \rangle \rangle -\frac{f_1}{2}\langle \dist f_2, \dist |\dist \phi|^2\rangle \right)\di \meas \nonumber \\
&=\int_X\left(-f_1\langle \dist \phi, \dist \langle \dist f_2, \dist \phi \rangle \rangle-f_1\langle \dist \phi, \dist f_2\rangle \Delta \phi  +f_1\langle \dist \phi, \dist \langle \dist f_2, \dist \phi \rangle \rangle -\frac{f_1}{2}\langle \dist f_2, \dist |\dist \phi|^2\rangle\right) \di \meas \nonumber \\
&=\int_X\left\langle -\Delta \phi \dist \phi-\frac{1}{2}\dist |\dist \phi|^2, f_1\dist f_2\right\rangle \di \meas
\end{align}
which completes the proof.
\end{proof}
Finally let us prove a characterization of compact non-collapsed $\RCD(K, N)$ spaces:
\begin{theorem}\label{thm:0}
Let $(X, \dist, \meas)$ be a compact $\RCD(K, N)$ space with $n=\dim_{\dist, \meas}(X)$.
Then the following two conditions are equivalent:
\begin{enumerate}
\item[(1)] $g_X \in D(\nabla^*)$ holds with $\nabla^*g_X=0$.
\item[(2)] It holds that 
\begin{equation}
\meas =\frac{\meas (X)}{\mathcal{H}^n(X)}\mathcal{H}^n
\end{equation}
holds and that $(X, \dist, \mathcal{H}^n)$ is a non-collapsed $\RCD(K, n)$ space.
\end{enumerate}
\end{theorem}
\begin{proof}
Let us first prove the implication from (1) to (2). Assume that (1) holds. Then for any $f_i \in \mathrm{Test}F(X, \dist, \meas)(i=1, 2)$ we have
\begin{equation}
\int_X\langle g_X, \nabla (f_1\dist f_2)\rangle \di \meas =0,
\end{equation}
thus by (\ref{eq:leib})
\begin{equation}
\int_X\langle \dist f_1, \dist f_2\rangle \di \meas =-\int_Xf_1\mathrm{tr}(\mathrm{Hess}_{f_2})\di \meas.
\end{equation}
Therefore we have $\Delta f_2=\mathrm{tr}(\mathrm{Hess}_{f_2})$. Then applying \cite[Thm.3.12]{BrueSemola} with \cite[Thm.1.12]{Gigli} yields $N=\dim_{\dist, \meas}(X)$. Thus we have (2) because of Theorem \ref{wnon}.

The converse implication is a direct consequence of (\ref{eq:leib}) and Theorem \ref{thm:hanresult}. Thus we conclude.
\end{proof}
\subsection{Convergence of tensors}
Let us recall the following well-known compactness result for a sequence of $\RCD(K, N)$ spaces (see \cite[Thm.6.11]{AmbrosioGigliSavare13}, \cite[Thm.5.3.22]{ErbarKuwadaSturm}, \cite[Thm.7.2]{GigliMondinoSavare13}, \cite[Thm.5.19]{LottVillani}, and \cite[Thm.4.20]{Sturm06a}).
\begin{theorem}[Compactness]\label{hohohonoho}
Let $(X_i, \dist_i, x_i, \meas_i)$ be a sequence of pointed $\RCD(K, N)$ spaces with 
\begin{equation}
0<\inf_i\meas_i(B_1(x_i))\le \sup_i \meas_i(B_1(x_i))<\infty. 
\end{equation}
Then after passing to a subsequence, there exists a pointed $\RCD(K, N)$ space $(X, \dist, x, \meas)$ such that $(X_i, \dist_i, x_i, \meas_i)$ pmGH-converge to $(X, \dist, x, \meas)$.
\end{theorem}
Let us fix $R \in (0, \infty]$ and a pmGH convergent sequence of pointed $\RCD(K, N)$ spaces:
\begin{equation}\label{jju}
(X_i, \dist_i, x_i, \meas_i) \stackrel{\mathrm{pmGH}}{\to} (X, \dist, x, \meas).
\end{equation}
In this setting it is well-defined that a sequence $f_i \in L^p(B_R(x_i), \meas_i)$ \textit{$L^p$-strongly/weakly converge to $f \in L^p(B_R(x), \meas)$ on $B_R(x)$} for $p \in [1, \infty)$. Note that $B_R(x)=X$ when $R=\infty$. We skip the precise definition (because we omitted the pmGH convergence. A typical example is that $1_{B_r(y_i)}$ $L^p$-strongly converge to $1_{B_r(y)}$ on $B_R(x)$ for any $r \in (0, \infty)$ and any $y_i \in X_i \to y \in X$). 
See \cite{AmbrosioHonda, AmbrosioHonda2, AmbrosioStraTrevisan, Ding, GigliMondinoSavare13, Honda2, KuwaeShioya}.

Instead of giving that, we give the definition of $L^2$-convergence of tensors as follows. Let us denote by $\nabla_i, \Delta_i$ and etc, the gradient operator, Laplacian etc for $(X_i, \dist_i, \meas_i)$ (such notations will be used later immediately). 
\begin{definition}[Convergence of tensors]
We say that a sequence $T_i \in L^2((T^*)^{\otimes 2}(B_R(x_i), \dist_i, \meas_i))$ \textit{$L^2$-weakly converge to $T \in L^2((T^*)^{\otimes 2}(B_R(x), \dist, \meas))$ on $B_R(x)$} if the following two conditions are satisfied.
\begin{enumerate}
\item $\sup_i\|T_i\|_{L^2(B_R(x_i))}<\infty$ holds.
\item We see that $\langle T_i, \dist f_{1, i} \otimes \dist f_{2, i}\rangle$  $L^2$-weakly converge to $\langle T, \dist f_1 \otimes \dist f_2\rangle$ on $B_R(x)$ whenever $f_{j, i} \in \mathrm{Test}F(X_i, \dist_i, \meas_i)$ $L^2$-strongly converge to $f_j \in \mathrm{Test}F(X, \dist, \meas)$ with 
\begin{equation}
\sup_{i, j}\left(\|f_{j, i}\|_{L^{\infty}(X_i)}+\|\nabla_i f_{j, i}\|_{L^{\infty}(X_i)} +\|\Delta_if_{j, i}\|_{L^2(X_i)}\right)<\infty.
\end{equation}
\end{enumerate}
Moreover we say that \textit{$T_i$ $L^2$-strongly converge to $T$ on $B_R(x)$} if it is an $L^2$-weak convergent sequence on $B_R(x)$ with $\limsup_{i \to \infty}\|T_i\|_{L^2(B_R(x_i))}\le \|T\|_{L^2(B_R(x))}$ holds. 
\end{definition}
Compare with \cite[Def.5.18 and Lem.6.4]{AHPT} (see also \cite[Thm.10.3]{AmbrosioHonda} and \cite[Prop.3.64]{Honda2}).
Next let us recall the definition of $H^{1, 2}$-strong convergence:
\begin{definition}[$H^{1, 2}$-strong convergence]\label{dettrrrss}
We say that a sequence of $f_i \in H^{1, 2}(B_R(x_i), \dist_i, \meas_i)$ \textit{$H^{1, 2}$-strongly converge to $f \in H^{1, 2}(B_R(x), \dist, \meas)$ on $B_R(x)$} if $f_i$ $L^2$-strongly converge to $f$ on $B_R(x)$ with $\lim_{i \to \infty}\|\nabla f_i\|_{L^2(B_R(x_i))}=\|\nabla f\|_{L^2(B_R(x))}$.
\end{definition}
In connection with Definition \ref{dettrrrss}, we introduce a Rellich type compactness result with respect to mGH-convergence (see \cite[Thm.6.3]{GigliMondinoSavare13} and \cite[Thm.7.4]{AmbrosioHonda})
\begin{theorem}[Convergence of gradient operators]\label{bbbg}
If a sequence $f_i \in H^{1, 2}(B_R(x_i), \dist_i, \meas_i)$ satisfies $\sup_i\|f_i\|_{H^{1, 2}}<\infty$, then after passing to a subsequence, there exists $f \in H^{1, 2}(B_R(x), \dist, \meas)$ such that
\begin{equation}
\liminf_{i \to \infty}\|\nabla f_{i}\|_{L^2(B_R(x_i))} \ge \|\nabla f\|_{L^2(B_R(x))}
\end{equation}
holds and that 
$f_i$ $L^2$-strongly converge to $f$ on $B_r(x)$ for any $r \in (0, R)$. Moreover if in addition $f_i$ $H^{1, 2}$-strongly converge to $f$ on $B_r(x)$ for some $r \in (0, R]$, then $|\nabla f_i|^2$ $L^1$-strongly converge to $|\nabla f|^2$ on $B_r(x)$.
\end{theorem}
Note that in Theorem \ref{bbbg} if $R<\infty$, then $L^2$-strong convergence of $f_i$ to $f$ is satisfied on $B_R(x)$, which is justified by using the Sobolev embedding theorem $H^{1, 2} \hookrightarrow L^{2N/(N-2)}$. See \cite[Thm.4.2]{AmbrosioHonda2}.
Based on Theorem \ref{bbbg}, we can easily check the following by an argument similar to the proof of \cite[Thm.10.3]{AmbrosioHonda} (see also \cite[Def.5.18 and Lem.6.4]{AHPT} and \cite[Prop.3.64]{Honda2}).  
\begin{proposition}[Lower semicontinuity of $L^2$-norms]\label{prop:low}
A sequence $T_i \in L^2((T^*)^{\otimes 2}(B_R(x_i), \dist_i, \meas_i))$ $L^2$-weakly converge to $T \in L^2((T^*)^{\otimes 2}(B_R(x), \dist, \meas))$ on $B_R(x)$, then it holds that
\begin{equation}
\liminf_{i \to \infty}\|T_i\|_{L^2(B_R(x_i))}\ge \|T\|_{L^2(B_R(x))}.
\end{equation}
\end{proposition}
The convergence of the heat flows with respect to (\ref{jju}) is discussed in \cite[Thm.5.7]{GigliMondinoSavare13}  (more precisely they discussed it in more general setting, for $\CD(K, \infty)$ spaces under pmG-convergence). As a corollary, it is proved in \cite[Thm.7.8]{GigliMondinoSavare13} that the following spectral convergence result holds, which will play a key role later (see \cite[Thm.7.3 and 7.9]{CheegerColding3} for Ricci limit spaces. Compare with \cite[Prop.3.3]{AmbrosioHonda2}).
\begin{theorem}[Spectral convergence]\label{thm:spectral}
If $(X, \dist)$ is compact, 
then 
\begin{equation}
\lambda_j(X_i, \dist_i, \meas_i) \to \lambda_j(X, \dist, \meas),\quad \forall j.
\end{equation}
Moreover for any $\phi_j \in D(\Delta)$ with $\Delta \phi_j+\lambda_j(X, \dist, \meas)\phi_j=0$, there exists a sequence of $\phi_{j, i} \in D(\Delta_i)$ such that $\Delta_i\phi_{j, i}+\lambda_j(X_i, \dist_i, \meas_i)\phi_{j, i}=0$ holds and that $\phi_{j, i}$ $H^{1, 2}$-strongly converge to $\phi_j$ on $X$
\end{theorem}
Let us recall the following stability results proved in \cite[Thm.4.4]{AmbrosioHonda2}.
\begin{theorem}[Stability of Laplacian on balls]\label{spectral2}
Let $f_i \in D(\Delta_i, B_R(x_i))$ satisfy 
$$
\sup_i(\|f_i\|_{H^{1, 2}(B_R(x_i))}+\|\Delta_i f_i\|_{L^2(B_R(x_i))})<\infty,
$$
and let us assume that $f_i$ $L^2$-strongly converge to $f \in L^2(B_R(x), \meas)$ on $B_R(x)$ (so that, by Theorem~\ref{bbbg}, 
$f \in H^{1, 2}(B_R(x), \dist, \meas)$). Then we have the following. 
\begin{enumerate}
\item[(1)] $f \in D(\Delta, B_R(x))$.
\item[(2)] $\Delta_if_i$ $L^2$-weakly converge to $\Delta f$ on $B_R(x)$. 
\item[(3)] $f_i$ $H^{1, 2}$-strongly converge to $f$ on $B_r(x)$ 
for any $r<R$. 
\end{enumerate}
\end{theorem}
Note that in Theorem \ref{spectral2} if $R=\infty$, then the $H^{1, 2}$-strong convergence of $f_i$ to $f$ is satisfied on $B_R(x)=X$. See \cite[Cor.10.4]{AmbrosioHonda}.

Finally let us mention the following result, where (\ref{3433}) is already proved in \cite[Thm.1.5]{Kita} by a different way (see also \cite[Rem.5.20]{AHPT} and \cite[Prop.3.78]{Honda2}).
\begin{proposition}[$L^2_{\mathrm{loc}}$-weak convergence of Riemannian metrics]\label{weakriem}
Assume $R<\infty$. Then $g_{X_i}$ $L^2$-weakly converge to $g_X$ on $B_R(x)$. In particular we see that
\begin{equation}\label{3433}
\liminf_{i \to \infty}\dim_{\dist_i, \meas_i}(X_i) \ge \dim_{\dist, \meas}(X)
\end{equation}
holds and that $g_{X_i}$ $L^2$-strongly converge to $g_X$ on $B_R(x)$  if and only if $\dim_{\dist_i, \meas_i}(X_i) = \dim_{\dist, \meas}(X)$ holds for any sufficiently large $i$.
\end{proposition}
\begin{proof}
The desired $L^2$-weak convergence is a direct consequence of Theorems \ref{bbbg} and \ref{spectral2}. Moreover Propositions \ref{Riemdef} and \ref{prop:low} yield
\begin{align}
\liminf_{i \to \infty}\dim_{\dist_i, \meas_i}(X_i) &=\liminf_{i \to \infty}\frac{1}{\meas_i(B_R(x_i))}\int_{B_R(x_i)}|g_{X_i}|^2\di \meas_i \nonumber \\
&\ge \frac{1}{\meas (B_R(x))}\int_{B_R(x)}|g_X|^2\di \meas=\dim_{\dist, \meas}(X).
\end{align}
The remaining statement is trivial from this observation.
\end{proof}
\subsection{Splitting theorem via splitting map}
We say that a map $\gamma$ from $\mathbb{R}$ to a metric space $(Z, \dist_Z)$ is a \textit{line} if it is an isometric embedding as metric spaces, that is, $\dist_Z(\gamma(s), \gamma(t))=|s-t|$ holds for all $s,t \in \mathbb{R}$. Then the \textit{Busemann function of $\gamma$}, $b_{\gamma}:Z \to \mathbb{R}$ is defined by
\begin{equation}
b_{\gamma}(x):=\lim_{t\to \infty}\left(t-\dist_Z(\gamma(t), x)\right).
\end{equation} 
Let us introduce an important result on the $\RCD$ theory, so-called the \textit{splitting theorem}, proved in \cite[Thm.1.4]{Gigli13} (see also \cite[Lem.1.21]{ABS}):
\begin{theorem}[Splitting theorem]\label{splitting}
Let $(X, \dist, \meas)$ be an $\RCD(0, N)$ space and let $x \in X$. Assume that the following (1) or (2) holds.
\begin{enumerate}
\item There exist lines $\gamma_i:\mathbb{R}\to X (i=1, 2, \ldots, k)$ such that $\gamma_i(0)=x$ and
\begin{equation}
\int_{B_1(x)}b_{\gamma_i}b_{\gamma_j}\di \meas =0,\quad \forall i \neq j
\end{equation}
are satisfied. 
\item There exist harmonic functions $f_i:X \to \mathbb{R} (i=1, 2, \ldots, k)$ such that $f_i(x)=0$ and $\langle \dist f_i, \dist f_j\rangle \equiv \delta_{ij}$ are satisfied.
\end{enumerate}
Let us put $\phi_i:=b_{\gamma_i}$ if (1) holds, $\phi_i:=f_i$ if (2) holds. 
Then there exist a pointed $\RCD(K, N-k)$ space $(Y, \dist_Y, y, \meas_Y)$ and an isometry
\begin{equation}
\Phi:(X, \dist, x, \meas) \to \left(\mathbb{R}^k\times Y, \sqrt{\dist_{\mathbb{R}^k}^2+\dist_Y^2}, (0_k, y), \mathcal{H}^k\otimes \meas_Y \right)
\end{equation}
such that $\phi_i \equiv \pi_i\circ \Phi$ holds, where $\pi_i:\mathbb{R}^k \times Y \to \mathbb{R}$ is the projection to the $i$-th $\mathbb{R}$ of the Euclidean factor $\mathbb{R}^k$.
\end{theorem}
Note that quantitative versions of Theorem \ref{splitting} are also justified. See \cite[Thm.1.5]{Gigli13} and \cite[Prop.3.7 and 3.9]{BPS2} (see also \cite[Prop.1.4 and 1.5]{BPS}). 
Let us end this subsection by giving the following result which is a variant of quantitative versions of Theorem \ref{splitting}, which will be used later.
\begin{theorem}\label{highregular}
Let $\delta \in (0, 1)$, let $C, L \in (0, \infty)$, let $k \in \mathbb{N}$, let $(X, \dist, x, \meas)$ be a pointed $\RCD(-\delta, N)$ space with $n=\dim_{\dist, \meas}(X)$. Assume that there exist $\phi_i \in D(\Delta, B_L(x))(i=1, 2, \ldots, k)$ such that  $|\nabla \phi_i| \le C$ on $B_L(x)$, 
\begin{equation}
\frac{1}{\meas (B_1(x))}\int_{B_L(x)}(\Delta \phi_i)^2\di \meas <\delta,
\end{equation}
and
\begin{equation}
\frac{1}{\meas (B_{1}(x))}\int_{B_{L}(x)}\left| \sum_{i=1}^k\dist \phi_i \otimes \dist \phi_i-g_X\right| \di \meas<\delta
\end{equation}
are satisfied.
%\item We have for any $i \in \{1, 2, \ldots, k+l\}$,
%\begin{equation}
%\frac{1}{\meas (B_1(x))}\int_{B_L(x)}(\Delta \phi_i)^2\di \meas <\delta.
%\end{equation}
%\item We have for all $i, j \in \{1, 2, \ldots, k\}$,
%\begin{equation}
%\frac{1}{\meas (B_1(x))}\int_{B_L(x)}\left| \langle \nabla \phi_i, \nabla \phi_j\rangle  -\delta_{ij}\right| \di \meas <\delta.
%\end{equation}
%\item We have for all $i \in \{1, 2, \ldots, k+l\}, j \in \{k+1, \ldots, k+l\}$
%\begin{equation}
%\frac{1}{\meas (B_1(x))}\int_{B_L(x)}\left| \langle \nabla \phi_i, \nabla \phi_j\rangle \right| \di \meas <\delta.
%\end{equation}
%\end{enumerate}
Then we see that $(X, \dist, x, \meas(B_1(x))^{-1}\meas)$ is $\Psi(\delta, L^{-1};N, k, C)$-pmGH close to $(\mathbb{R}^n, \dist_{\mathbb{R}^n}, 0_n, \omega_n^{-n}\mathcal{H}^n)$ and that $\Phi=(\phi_1, \ldots, \phi_{k}):B_1(x) \to \mathbb{R}^{k}$ gives a $\Psi(\delta, L^{-1};N, k, C)$-GH approximation map to the image.
\end{theorem}
\begin{proof}
The proof is done by contradiction. If not, then there exist $\epsilon_0 \in (0, 1)$, a sequence of $\delta_i \to 0^+$, a sequence of $L_i \to \infty$, a sequence of pointed $\RCD(-\delta_i, N)$ spaces $(X_i, \dist_i, x_i, \meas_i)$ with $\meas_i(B_1(x_i))=1$, a sequence of functions $\phi_{j, i} \in D(\Delta_i, B_{L_i}(x_i))(j=1, 2, \ldots, k)$ with $|\nabla_i\phi_{j, i}|\le C$ on $B_{L_i}(x_i)$ such that the following holds.
\begin{enumerate}
\item We have 
\begin{equation}\label{eeqqtt}
\int_{B_{L_i}(x_i)}\left| \sum_{j=1}^k\dist_i \phi_{j, i} \otimes \dist_i \phi_{j, i}  -g_{X_i}\right| \di \meas_i + \sum_{j=1}^k\int_{B_{L_i}(x_i)}(\Delta_i \phi_{j, i})^2\di \meas_i \to 0.
\end{equation}
\item For any $i \in \mathbb{N}$, one of the following two conditions (a) and (b) is satisfied:
\begin{enumerate}
\item $(X_i, \dist_i, x_i, \meas_i)$ is not $\epsilon_0$-pmGH close to $(\mathbb{R}^{n_i}. \dist_{\mathbb{R}^{n_i}}, 0_{n_i}, \omega_{n_i}^{-1}\mathcal{H}^{n_i})$, where $n_i=\dim_{\dist_i, \meas_i}(X_i)$.
\item $\Phi_i=(\phi_{1, i},\ldots, \phi_{k, i}):B_1(x_i) \to \mathbb{R}^k$ is not an $\epsilon_0$-GH-approximation to the image.
\end{enumerate}
\end{enumerate}
By Theorems \ref{hohohonoho}, \ref{bbbg} and \ref{spectral2}, with no loss of generality we can assume that $\dim_{\dist_i, \meas_i}(X_i)$ does not depend on $i$ (thus we denote it by $n$) and that there exist an $\RCD(0, N)$ space $(X, \dist, \meas)$ and harmonic functions $\phi_j$ on $(X, \dist, \meas)$ such that $\|\nabla \phi_j\|_{L^{\infty}}<\infty$ holds, that $(X_i, \dist_i, x_i, \meas_i)$ pmGH-converge to $(X, \dist, x, \meas)$ and that $\phi_{j, i}$ $H^{1, 2}$-strongly converge to $\phi_j$ on $B_R(x)$ for any $R \in (0, \infty)$. 
Note that (\ref{eeqqtt}) implies 
\begin{equation}\label{eeqqtt4}
\int_{B_{R}(x_i)}\left| \sum_{i=1}^k\dist_i \phi_{j, i} \otimes \dist_i \phi_{l, i}  -g_{X_i}\right|^2 \di \meas_i \to 0,\quad \forall R\in (0, \infty)
\end{equation}
because of $|\nabla_i\phi_{j, i}|\le C$.
In particular Propositions \ref{prop:low} and \ref{weakriem} yield that for any $R \in (0, \infty)$,
\begin{equation}
\int_{B_{R}(x)}\left| \sum_{j=1}^k\dist \phi_j \otimes \dist \phi_j-g_X\right|^2\di \meas \le \liminf_{i \to \infty}\int_{B_{R}(x_i)}\left| \sum_{i=1}^k\dist_i \phi_{j, i} \otimes \dist_i \phi_{l, i}  -g_{X_i}\right|^2 \di \meas_i=0.
\end{equation}
Thus we have
\begin{equation}\label{hyttytyty}
\sum_{j=1}^k\dist \phi_j \otimes \dist \phi_j=g_X
\end{equation}
on $X$.
In particular
\begin{equation}\label{eq:constant}
\sum_{j=1}^k|\dist \phi_j|^2=\left\langle \sum_{j=1}^k\dist \phi_j \otimes \dist \phi_j, g_X\right\rangle=|g_X|^2=l,
\end{equation}
where $l=\dim_{\dist, \meas}(X)$.

On the other hand applying the (localized) Bochner inequality (\ref{eq:bochner}) yields
\begin{equation}\label{eq:locboch}
\frac{1}{2}\int_X|\nabla \phi_j|^2\Delta \phi \di \meas \ge \int_X\phi |\mathrm{Hess}_{\phi_j}|^2\di \meas
\end{equation}
for any $\phi \in D(\Delta)$ with compact support, $\phi \ge 0$ and $\Delta \phi \in L^{\infty}(X, \meas)$. In particular taking the sum with respect to $j$ in (\ref{eq:locboch}) with (\ref{eq:constant}) shows
\begin{equation}
0=\frac{1}{2}\int_X\sum_{j=1}^k|\nabla \phi_j|^2\Delta \phi \di \meas \ge \int_X\sum_{j=1}^k\phi |\mathrm{Hess}_{\phi_j}|^2\di \meas \ge 0.
\end{equation}
Thus $\mathrm{Hess}_{\phi_j}=0$. In particular $\langle \nabla \phi_i, \nabla \phi_j \rangle$ is constant for all $i, j$ because for any bounded open subset $U$ of $X$, we have $\langle \nabla \phi_i, \nabla \phi_j\rangle|_U \in H^{1, 2}(U, \dist, \meas)$ with
\begin{equation}\label{acf}
|\nabla \langle \nabla \phi_i, \nabla \phi_j\rangle| \le |\nabla \phi_j| \cdot |\mathrm{Hess}_{\phi_i}| + |\nabla \phi_i| \cdot |\mathrm{Hess}_{\phi_j}|=0.
\end{equation}
Then applying Theorems \ref{splitting} for $\phi_j$ shows that $(X, \dist, \meas)$ is isometric to $(\mathbb{R}^l, \dist_{\mathbb{R}^l}, \omega_l^{-1}\mathcal{H}^l)$ (because the (LHS) of (\ref{hyttytyty}) consists of the Riemannian metric of the Euclidean factor coming from $\phi_j$), and that the map $\Phi=(\phi_1,\ldots, \phi_k):X \to 
\mathbb{R}^k$ is an isometric embedding as metric spaces (because each $\phi_j$ is a linear function). Since $\phi_{i, j}$ converge uniformly to $\phi_j$ on $B_r(0_k)$ for any $r \in (0, \infty)$, for any $\epsilon \in (0, \infty)$, $\Phi_i:B_1(x_i) \to \mathbb{R}^k$ gives an $\epsilon$-GH-approximation to the image for any sufficiently large $i$. 
Therefore it is enough to prove $l=n$ in order to get a contradiction.

Combining Proposition \ref{weakriem} with (\ref{eeqqtt4}) yields that $g_{X_i}$ $L^2$-strongly converge to $g_{\mathbb{R}^l}$ on $B_{1}(0_l)$.
Thus we have
\begin{equation}
l=\frac{1}{\mathcal{H}^l(B_{1}(0_l))}\int_{B_{1}(0_l)}|g_{\mathbb{R}^l}|^2\di \mathcal{H}^l=\lim_{i \to \infty}\frac{1}{\meas_i (B_{1}(x_i))}\int_{B_{1}(x_i)}|g_{X_i}|^2\di \meas_i=n.
\end{equation}
%Let $J:=\{j; |\nabla \phi_{j, 0}| \neq 0\}$.
%Then Theorem \ref{splitting} yields that for any $j \in J$ there exists a line $\gamma_j:\mathbb{R} \to Y$ such that %$c_jb_{\gamma_i}\equiv \phi_{j, 0}$ for some $c_j \in (0, \infty)$. In particular the (LHS) of (\ref{eq:100}) consists of the %Riemannian metric of the Euclidean factor of $(Y, \dist_Y, \meas_Y)$ coming from $\{\gamma_j\}_{j \in J}$ which shows %that $(Y, \dist_Y)$ is isometric to an Euclidean space whose dimension is at most $k+1$.
\end{proof}
\begin{remark}\label{remarkremark}
As an immediately consequence of Theorem \ref{highregular} (although it may be independent of our interest in the paper), it is easy to see that the following holds:
\begin{itemize}
\item Let $U$ be an open subset of an $\RCD(K, N)$ space $(X, \dist, \meas)$ with $n=\dim_{\dist, \meas}(X)$, let $x \in U$, let $\Phi=(\phi_1, \ldots, \phi_k):U \to \mathbb{R}^k$ be a locally Lipschitz map.
Assume that $x$ is a harmonic point of each $\phi_i$ and that 
\begin{equation}
\lim_{r \to 0^+}\frac{1}{\meas (B_r(x))}\int_{B_r(x)}\left|\sum_{i=1}^k\dist \phi_i \otimes \dist \phi_i-g_X\right|\di \meas=0
\end{equation}
holds. Then $x \in \mathcal{R}_n$.
\end{itemize}
See \cite[Def.5.2]{AHPT} for the definition of \textit{harmonic points} of $H^{1, 2}$ functions.

%Let $(X, \dist, x, \meas)$ be a pointed $\RCD(K, N)$ space, let $U$ be an open subset of $X$ with $x \in U$, let $z \in B_R(x)$ and let $f \in \mathrm{Lip}(U, \dist)$. We say that \textit{$z$ is a harmonic point of $f$} if the following holds:
%\begin{enumerate}
%\item The limit
%\begin{equation}
%\lim_{r \to 0^+}\frac{1}{\meas (B_r(z))}\int_{B_r(z)}|\nabla f|^2\di \meas
%\end{equation}
%exists. 
%\item For any $t_i \to 0^+$ and any pmGH-convergent sequence:
%\begin{equation}
%(X,t_{i}^{-1} \dist,z, \meas(B_{t_{i}}(z))^{-1} \meas) \stackrel{\mathrm{pmGH}}{\to} (Y, \dist_Y, y, \meas_Y),
%\end{equation}
%after passing to a subsequence, there exists a Lipschitz harmonic function $\hat{f}$ on $Y$ such that the rescaled %functions $f_{t_{i},z}$ $H^{1, 2}$-strongly converge to $\hat{f}$ on $B_R(y)$ for any $R \in (0, \infty)$, where $f_{z,t}$ is %defined by
%\begin{equation}
%f_{t, z}:=\frac{1}{t}\left( f-\fint_{B_{t}(z)}f\di\meas \right)
%\end{equation}
%on $(X,t^{-1} \dist,\meas(B_{t}(z))^{-1} \meas)$.
%\end{enumerate}
Instead of giving the precise definition, let us emphasize that 
this notion is closely related to the differentiability of the function at a given point and that for any $H^{1, 2}$ function, the set of all harmonic points of the function has full measure. For instance in the Riemannian case $(X, \dist, \meas)=(M^n, \dist_g, \mathrm{vol}_g)$ with $f \in C^1(M^n)$, every point $x \in M^n$ is a harmonic point of $f$. On the other hand if $f(x)=|x|$ on $\mathbb{R}^n$, then $0_n$ is not a harmonic point of $f$. 

Let us consider the canonical inclusion $\iota: \overline{B}_1(0_n) \hookrightarrow \mathbb{R}^n$ with $e_n=(0, 0,\ldots, 0, 1) \in \overline{B}_1(0_n)$. Note that $(\overline{B}_1(0_n), \dist_{\mathbb{R}^n}, \mathcal{H}^n)$ is a compact non-collapsed $\RCD(0, n)$ space with
\begin{equation}
\sum_{i=1}^n\dist x_i \otimes \dist x_i=g_{\overline{B}_1(0_n)}.
\end{equation}
Then it is worth pointing out that $e_n$ is a harmonic point of the coordinate function $x_i$ for any $i \in \mathbb{N}_{\le n-1}$. However $e_n$ is not a harmonic point of $x_n$.
\end{remark}
\section{Isometric immersion}\label{immersion}
Let us fix an $\RCD(K, N)$ space with $n=\dim_{\dist, \meas}(X)$. 
\begin{definition}[Pull-back metric]
Let $A$ be a Borel subset of $X$ with $\meas (A)>0$, and let $\Phi=(\phi_1, \phi_2,\ldots, \phi_k):A \to \mathbb{R}^k$ be a locally Lipschitz map, that is, for any $x \in A$ there exists $r \in (0, 1)$ such that $\Phi|_{A \cap B_r(x)}$ is Lipschitz. Define the \textit{pull-back (Riemannian semi) metric $\Phi^*g_{\mathbb{R}^k}$} by
\begin{equation}
\Phi^*g_{\mathbb{R}^k}:=\sum_{i=1}^k\dist \phi_i\otimes \dist \phi_i \in L^{0}((T^*)^{\otimes 2}(A, \dist, \meas)),
\end{equation}
where $\dist \phi_i$ is the restriction to $A$ of the exterior derivative $\dist \overline{\phi}_i$ of a Lipschitz extension $\overline{\phi}_i$ of $\phi_i|_{A \cap B_r(x)}$ to $X$ (that is, $\dist \phi_i=\dist \overline{\phi}_i|_A$) which is well-defined because of the locality of the minimal relaxed slope. 
Then we say that $\Phi$ is an \textit{isometric immersion} if $\Phi^*g_{\mathbb{R}^k}=g_X$ holds in $L^{0}((T^*)^{\otimes 2}(A, \dist, \meas))$.
\end{definition}
This definition is compatible with \cite[Prop.4.9]{AHPT} which gives the definition of pull-back (Riemannian semi) metric of a general Lipschitz map into a real separable Hilbert space. Note that if $\Phi$ is $C$-Lipschitz, then $\|\Phi^*g_{\mathbb{R}^k}\|_{L^{\infty}(A)}\le C^2n$.
\begin{definition}[Regular map]
Let $U$ be an open subset of $X$. Then a map $\Phi:=(\phi_1,\ldots, \phi_k):U \to \mathbb{R}^k$ is said to be \textit{regular on $U$} if each $\phi_i$ is in $D(\Delta, U)$ with $\Delta \phi_i \in L^{\infty}(U, \meas)$. 
\end{definition}
Note that thanks to regularity results proved in \cite[Thm.3.1]{AMS} and \cite[Thm.3.1]{Jiang}, independently, any regular map $\Phi:U \to \mathbb{R}$ is locally Lipschitz.
\begin{definition}[Locally uniformly $\delta$-isometric immersion]\label{delta}
Let $U$ be an open subset of $X$, let $\delta \in (0, \infty)$, let $\Phi:U \to \mathbb{R}^k$ be a locally Lipschitz map and let $A$ be a Borel subset of $U$.
We say that $\Phi$ is a \textit{locally uniformly $\delta$-isometric immersion on A} if for any $x \in A$ there exists $r_0 \in (0, 1)$ such that $B_{r_0}(x) \subset U$ and 
\begin{equation}
\frac{1}{\meas (B_r(y))}\int_{B_{\delta^{-1}r}(y)}\left| \Phi^*g_{\mathbb{R}^k}-g_X\right| \di \meas<\delta, \quad \forall y \in B_{r_0}(x) \cap A,\,\forall r \in (0, r_0)
\end{equation}  
are satisfied.
Moreover $\Phi$ is said to be a \textit{locally uniformly isometric immersion on $A$} if it is a locally uniformly $\delta$-isometric immersion on $A$ for any $\delta \in (0, 1)$.
\end{definition}
It is trivial that for a locally Lipschitz map $\Phi:U \to \mathbb{R}^k$ on open subset $U$ of $X$, $\Phi$ is a locally uniformly isometric immersion on $U$ if and only if $\Phi$ is an isometric immersion.
\begin{theorem}\label{yhug}
Let $U$ be an open subset of $X$,  let $A$ be a Borel subset of $U$, and let $\Phi=(\phi_1,\ldots, \phi_k):U \to \mathbb{R}^k$ be a regular map with 
\begin{equation}\label{bounde}
\|\nabla \phi_i\|_{L^{\infty}(U)} \le C.
\end{equation}
Assume that $\Phi$ is a locally uniformly $\delta$-isometric immersion on $A$.
Then we have the following:
\begin{enumerate}
\item The set $A$ is locally $\Psi(\delta; K, N, k, C)$-Reifenberg flat, that is, for any $x \in A$ there exists $r_1 \in (0, 1)$ such that 
\begin{equation}
\dist_{\mathrm{GH}}(B_r(y), B_r(0_n))<\Psi(\delta; K, N, k, C)r,\quad \forall r \in (0, r_1],\,\forall y\in B_{r_1}(x) \cap A
\end{equation}
holds.
\item $\Phi|_A$ gives a locally bi-Lipschitz embedding map from $A$ to $\mathbb{R}^k$ whenever $\delta$ is sufficiently small depending only on $K, N, k$ and $C$. More precisely, for any $x \in A$, there exists $r_2 \in (0, 1)$ such that $\Phi_{B_{r_2}(x)\cap A}$ and $(\Phi|_{B_{r_2}(x) \cap A})^{-1}$ are $(1+\Psi(\delta; K, N, k, C))$-Lipschitz maps. 
\end{enumerate}
\end{theorem}
\begin{proof}
Let $x\in A$, let $r_0$ be as in Definition \ref{delta} with
\begin{equation}
\max_i\|\Delta \phi_i\|_{L^{\infty}}\cdot \left(\int_0^1\sinh^{N-1} t\di t \right)^{-1} \le \frac{1}{\sqrt{r_0}},
\end{equation}
and let $y \in B_{r_0}(x) \cap A$.
For $r \in (0, r_0]$ let us denote by $\nabla_r, \Delta_r$ etc the gradient operator, the Laplacian, etc of the rescaled space 
\begin{equation}
(X_r, \dist_r, \meas_r):=(X, r^{-1}\dist, \meas (B_r(y))^{-1}\meas),
\end{equation}
which is an $\RCD(r^2K, N)$ space. Note that the rescaled functions $\phi_{r, i}:=r^{-1}\phi_i$ satisfy $|\nabla_r\phi_{r, i}|=|\nabla \phi| \le C$ and 
\begin{align}
\int_{B^{\dist_r}_L(y)}|\Delta_r\phi_{r, i}|^2\di \meas_r &=\int_{B^{\dist_r}_L(y)}r^2|\Delta \phi_i|^2\di \meas_r \nonumber \\
&\le \|\Delta \phi_i\|_{L^{\infty}(U)}^2r^2\meas_r(B^{\dist_r}_L(y)) \nonumber \\
&\le \|\Delta \phi_i\|_{L^{\infty}(U)}^2r^2 \int_0^L\sinh^{N-1} t\di t \cdot \left(\int_0^1\sinh^{N-1} t\di t \right)^{-1}\nonumber \\
&\le \|\Delta \phi_i\|_{L^{\infty}(U)}^2r\left(\int_0^1\sinh^{N-1} t\di t \right)^{-1} \le \sqrt{r},
\end{align}
where $L=L(N, r)>0$ is defined by satisfying
\begin{equation}
r^{-1}=\int_0^L\sinh^{N-1} t\di t
\end{equation}
and we used the Bishop-Gromov inequality as an $\RCD(-(N-1), N)$ space (because $r^2K \ge -(N-1)$ when $r$ is small).
Then applying Theorem \ref{highregular} for $(X_r, \dist_r, y, \meas_r)$ and $\phi_{r, i}$ yields that $(X, r^{-1}\dist, y, \meas (B_r(y))^{-1}\meas)$ is $\Psi(\delta, r;K, N, k, C)$-pmGH close to $(\mathbb{R}^n, \dist_{\mathbb{R}^n}, 0_n, \omega_n^{-1}\mathcal{H}^n)$ and that $\Phi:B_{r}(y) \to \mathbb{R}^k$ gives a $\Psi(\delta, r; K, N, k, C)r$-GH-approximation to the image. Thus we have (1) and
\begin{equation}\label{wws}
\left| |\Phi(z)-\Phi(w)|_{\mathbb{R}^{k}}-\dist (z, w)\right| \le \Psi (\delta, r; K, N, k, C)r,\quad \forall z, w \in B_{r}(y).
\end{equation}
In particular for all $z, w \in B_{r_0/4}(x)$ letting $r:=\dist(z, w)$ and $w=y$ in (\ref{wws}) shows
\begin{equation}
\left||\Phi(z)-\Phi (w)|_{\mathbb{R}^{k}}-\dist (z, w)\right| \le \Psi (\delta, r; K, N, k, C)\dist (z, w),
\end{equation}
that is,
\begin{equation}
(1-\Psi(\delta, r; K, N, k, C))\dist(z, w) \le |\Phi (z)-\Phi (w)|_{\mathbb{R}^{k}} \le (1+\Psi(\delta, r; K, N, k, C))\dist (z, w).
\end{equation}
Therefore $\Phi|_{B_{r_0/4}(x)}$ is a $(1+\Psi(\delta, r; K, N, k, C))$-Lipschitz embedding satisfying that $(\Phi|_{B_{r_0/4}(x)})^{-1}$ is $(1-\Psi(\delta, r; K, N, k, C))^{-1}$-Lipschitz, whenever $\delta, r$ are sufficiently small  depending only on $K, N, k$ and $C$. Thus we conclude.
\end{proof}
We are now in a position to prove Theorem \ref{thm:bilip}.

\textit{Proof of Theorem \ref{thm:bilip}.}

Applying Theorem \ref{yhug} as $A=U$ and arbitrary $\delta \in (0, 1)$ yields (1) and (2) of Theorem \ref{thm:bilip}. The remaining part, that $U$ is homeomorphic to an $n$-dimensional topological manifold without boundary, is justified by the Reifenberg flatness with the intrinsic Reifenberg theorem proved in \cite[Thm.A.1.2]{CheegerColding1}. $\,\,\,\,\,\square$

%\begin{theorem}[Local bi-Lipschitz embedding implies isometric immersion]\label{labt}
%Let $U$ be an open subset of $X$ and let $A$ be a Borel subset of $U$. Assume that there exist $\phi_i \in D(\Delta, U)(i=1, 2, \ldots, k)$ and $r_0 \in (0, 1)$ such that (\ref{bounde}) holds and that the map $\Phi:=(\phi_1, \phi_2,\ldots, \phi_k):U \to \mathbb{R}^k$ satisfies that there exists $r_0 \in (0, 1)$ such that for any $x \in \mathrm{Leb}(A)$, $\Phi|_{B_r(x)\cap A}$ is a bi-Lipschitz embedding satisfying that $\Phi|_{B_r(x)\cap A}$ and $(\Phi|_{B_r(x)\cap A})^{-1}$ are $(1+\delta)$-Lipschitz.
%Then
%\begin{equation}
%\left| \sum_{i=1}^k\dist \phi_i \otimes \dist \phi_i-g_X\right|<\Psi(\delta, r; K, N, C).
%\end{equation}
%\end{theorem}
%\begin{proof}
%\end{proof}
%\begin{proposition}[Approximation]
%\end{proposition}
%\begin{proof}
%\end{proof}
%\begin{remark}
%Although we assumed the $L^{\infty}_{\mathrm{loc}}$-bound on $\Delta \phi$ in Theorems \ref{yhug} and \ref{labt}, we will be able to improve this to the $L^p$-bound for some $p \in (N, \infty]$. See for instance the proof of \cite[Thm.5.2]{Honda3}. %However it is enough to consider the $L^{\infty}$-version when we discuss eigenmaps.
%\end{remark}

\section{Eigenmap}\label{sseewwe}
In this section we discuss eigenmaps on compact $\RCD(K, N)$ spaces.
\subsection{Definition}
Let us fix a compact $\RCD(K, N)$ space $(X, \dist, \meas)$.  
\begin{definition}[Eigenmap]
We say that a map $\Phi=(\phi_1, \ldots, \phi_k):X \to \mathbb{R}^k$ is a \textit{$k$-dimensional eigenmap} if each $\phi_i$ is an eigenfunction of $-\Delta$ for any $i$ or $\phi_i\equiv 0$, that is, $\Delta \phi_i+\lambda_i\phi_i\equiv 0$ for some $\lambda_i \in [0, \infty)$. Put $A(\Phi):=\{i; |\nabla\phi_i| \not \equiv 0\}$, 
\begin{equation}
\Lambda(\Phi):=\max_{i \in A(\Phi)}\lambda_i,\quad L(\Phi):=\min_{1 \le i \le k}\frac{1}{\meas (X)}\int_X\phi_i^2\di \meas.
\end{equation}
Moreover $\Phi$ is said to be \textit{irreducible} if each $\phi_i$ is not a constant function.
\end{definition}
\begin{theorem}\label{cor:1}
Any $k$-dimensional eigenmap $\Phi:X \to \mathbb{R}^{k}$ satisfies
\begin{equation}\label{kii}
\nabla^*\Phi^*g_{\mathbb{R}^{k}}=-\frac{1}{4}\dist \Delta |\Phi|_{\mathbb{R}^k}^2.
\end{equation}
\end{theorem}
\begin{proof}
It is enough to check (\ref{kii}) under assuming $k=1$.
Then since $\Delta |\Phi|^2=2\Phi \Delta \Phi +2|\dist \Phi|^2$, Proposition \ref{thm:3} shows
\begin{align}
\dist \Delta |\Phi|^2&=2\Delta \Phi \dist \Phi +2\Phi \dist \Delta \Phi +2\dist |\dist \Phi|^2 \nonumber \\
&=-4\lambda \Phi \dist \Phi +2\dist |\dist \Phi|^2 \nonumber \\
&=4\Phi \dist \Delta \Phi +2\dist |\dist \Phi|^2 \nonumber \\
&=-4\nabla^*(\dist \Phi \otimes \dist \Phi), 
\end{align}
where $\lambda$ denotes the eigenvalue of $\Phi$.
\end{proof}
\subsection{Compactness for eigenmaps}
Throughout this subsection let us fix an mGH-convergent sequence of compact $\RCD(K, N)$ spaces
\begin{equation}\label{uhb}
(X_i, \dist_i, \meas_i) \stackrel{\mathrm{mGH}}{\to} (X, \dist, \meas).
\end{equation}
Let us discuss a convergence of eigenmaps in the following sense.
\begin{definition}[Convergence of eigenmaps]\label{def:conveigen}
We say that a sequence of $k$-dimensional eigenmaps $\Phi_i:X_i \to \mathbb{R}^k$ \textit{converge} to a $k$-dimensional eigenmap $\Phi:X \to \mathbb{R}^k$ if $\pi_j\circ \Phi_i$ $L^2$-strongly converge to $\pi_j\circ \Phi$ on $X$ for any $j$, where $\pi_j:\mathbb{R}^k \to \mathbb{R}$ is the projection to the $j$-th $\mathbb{R}$, and $\Phi_i^*g_{\mathbb{R}^k}$ $L^2$-strongly converge to $\Phi^*g_{\mathbb{R}^k}$ on $X$.
\end{definition}
Let us fix $k \in \mathbb{N}$ and $k$-dimensional eigenmaps $\Phi_i=(\phi_{1, i}, \ldots, \phi_{k, i}):X_i \to \mathbb{R}^k$ below.
\begin{proposition}[Compactness for general eigenmaps]\label{ssx}
Assume 
\begin{equation}\label{eigcomp}
\sup_i\left( \left\||\Phi_i|_{\mathbb{R}^k}\right\|_{L^2}+\|\Phi_i^*g_{\mathbb{R}^k}\|_{L^1} +\Lambda (\Phi_i)\right)<\infty.
\end{equation}
Then after passing to a subsequence, there exists a $k$-dimensional eigenmap $\Phi: X \to \mathbb{R}^k$ such that $\Phi_i$ converge to $\Phi$ with
\begin{equation}\label{lowr}
\liminf_{i \to \infty}\Lambda(\Phi_i) \ge \Lambda (\Phi), \quad \lim_{i \to \infty}L(\Phi_i)= L(\Phi).
\end{equation}
\end{proposition}
\begin{proof}
By Theorems \ref{bbbg}, \ref{thm:spectral} and \ref{spectral2} with (\ref{eq:eigenest}) and (\ref{eigcomp}), after passing to a subsequence, with no loss of generality we can assume  that there exist $\lambda_j \in \mathbb{R}$ and $\phi_j \in D(\Delta)$ such that $\Delta \phi_j+\lambda_j\phi_j=0$ holds, that $\sup_i\|\nabla \phi_{j, i}\|_{L^{\infty}}<\infty$ holds and that $\phi_{j, i}$ $H^{1, 2}$-strongly converge to $\phi_j$. Then it is easy to check that $\Phi_i$ converge to a $k$-dimensional eigenmap $\Phi:=(\phi_1, \ldots, \phi_k):X \to \mathbb{R}^k$ and that (\ref{lowr}) holds.
\end{proof}
\begin{remark}
In Theorem \ref{ssx} it is essential to assume a uniform upper bound on $\Lambda (\Phi_i)$. In order to check this, let us consider a sequence of collapsing tori:
\begin{equation}
(X_r, \dist_r, \meas_r):=\left(\mathbb{S}^1\times \mathbb{S}^1(r), \sqrt{\dist_{\mathbb{S}^1}^2+\dist_{\mathbb{S}^1(r)}^2}, \frac{1}{4\pi^2r}\mathcal{H}^2\right) \stackrel{\mathrm{mGH}}{\to} \left(\mathbb{S}^1, \dist_{\mathbb{S}^1}, \frac{1}{2\pi}\mathcal{H}^1\right) =:(X, \dist, \meas)
\end{equation}
as $r \to 0^+$. Then for any $r \in (0, \infty)$, the canonical inclusion $\Phi_r:X_r \to \mathbb{R}^4$ is an irreducible $4$-dimensional eigenmap with $\Phi_r^*g_{\mathbb{R}^4}=g_{X_r}$. Note that $\Lambda (\Phi_r) \to \infty$ holds, that $L(\Phi_r) \to 0$ holds,  and that $g_{X_r}$ $L^2$-weakly converge to $g_X$, but it is not an $L^2$-strong convergence because of
\begin{equation}
\int_{X_r}|g_{X_r}|^2\di \meas_r\equiv 2>1\equiv \int_X|g_X|^2\di \meas.
\end{equation}
In particular $\Phi_r^*g_{\mathbb{R}^4}$ has no $L^2$-strong convergent subsequence as $r \to 0^+$. 
\end{remark}
\begin{corollary}[Compactness for irreducible eigenmaps]\label{ssx2}
Assume that each $\Phi_i$ is irreducible with
\begin{equation}
\sup_i\|\Phi_i^*g_{\mathbb{R}^k}\|_{L^1}<\infty
\end{equation}
and 
\begin{equation}
\inf_iL(\Phi_i)>0.
\end{equation}
Then it holds that (\ref{eigcomp}) is satisfied and that after passing to a subsequence, there exists a $k$-dimensional irreducible eigenmap $\Phi: X \to \mathbb{R}^k$ such that $\Phi_i$ converge to $\Phi$ with
\begin{equation}\label{nbgyt}
L(\Phi_i) \to L(\Phi), \quad \Lambda(\Phi_i) \to \Lambda(\Phi).
\end{equation}
\end{corollary}
\begin{proof}
Since
\begin{align}
\int_{X_i}|\phi_{j, i}|^2\di \meas_i&\le \lambda_1(X_i, \dist_i, \meas_i)^{-1}\int_{X_i}|\dist \phi_{j, i}|^2\di \meas_i \nonumber \\
&\le \lambda_1(X_i, \dist_i, \meas_i)^{-1}\int_{X_i}\langle \Phi_i^*g_{\mathbb{R}^k}, g_{X_i}\rangle \di \meas_i \nonumber \\
&\le \lambda_1(X_i, \dist_i, \meas_i)^{-1}\sqrt{\dim_{\dist_i, \meas_i}(X_i)}\int_{X_i}|\Phi_i^*g_{\mathbb{R}^k}|\di \meas_i,
\end{align} 
we have $\sup_i\|\Phi_i\|_{L^2}<\infty$.
Let us denote by $\lambda_{j, i}$ the eigenvalue of $\phi_{j, i}$.
Then since
\begin{align}
\lambda_{j, i}&=\left(\int_{X_i}|\dist \phi_{j, i}|^2\di \meas_i\right) \cdot \left( \int_{X_i}|\phi_{j, i}|^2 \di \meas_i\right)^{-1} \nonumber \\
&\le \sqrt{\dim_{\dist_i, \meas_i}(X_i)}\int_{X_i}|\Phi_i^*g_{\mathbb{R}^k}|\di \meas_i\cdot \left( \inf_iL(\Phi_i) \right)^{-1},
\end{align}
we have $\sup_i\Lambda (\Phi_i)<\infty$. Thus applying Proposition \ref{ssx} yields that after passing to a subsequence, there exists a $k$-dimensional eigenmap $\Phi:X \to \mathbb{R}^k$ such that $\Phi_i$ converge to $\Phi$.
The irreducibility of $\Phi$ comes from the fact that $\lambda_1(X_i, \dist_i, \meas_i) \to \lambda_1(X, \dist, \meas)$ by Theorem \ref{thm:spectral}. Then it is easy to see (\ref{nbgyt}).
\end{proof}
Finally let us mention the following approximation result. 
\begin{proposition}[Approximation]\label{prop:appl}
Let $\hat{\Phi}=(\hat{\phi}_1, \ldots, \hat{\phi}_k):X \to \mathbb{R}^k$ be a $k$-dimensional (irreducible, respectively) eigenmap. Then there exists a sequence of $k$-dimensional (irreducible, respectively) eigenmaps $\hat{\Phi}_i:X_i \to \mathbb{R}^k$ such that $\hat{\Phi}_i$ converge to $\hat{\Phi}$ with 
\begin{equation}
\sup_i\|\hat{\Phi}_i^*g_{\mathbb{R}^k}\|_{L^{\infty}}<\infty, \quad \Lambda (\hat{\Phi}_i) \to \Lambda (\hat{\Phi}), \quad L(\hat{\Phi}_i) \to L(\hat{\Phi}).
\end{equation}
\end{proposition}
\begin{proof}
Let us find $\lambda_j \in [0, \infty)$ with $\Delta \hat{\phi}_j+\lambda_j\hat{\phi}_j\equiv 0$.
By Theorem \ref{thm:spectral}, for any $j$, we can find a convergent sequence $\lambda_{j, i} \to \lambda_j$ and sequence of $\hat{\phi}_{j, i} \in D(\Delta^i)$ with $\Delta_i \hat{\phi}_{j, i}+\lambda_{j, i}\hat{\phi}_{j, i}\equiv 0$ such that $\hat{\phi}_{j, i}$ $H^{1, 2}$-strongly converge to $\hat{\phi}_j$. Then the eigenmaps $\hat{\Phi}_i:=(\hat{\phi}_{1, i},\ldots, \hat{\phi}_{k, i}):X_i \to \mathbb{R}^k$ satisfy the desired properties.
\end{proof}
\section{Isometric immersion via eigenmap}\label{secproof}
The main purpose of this section is to prove Theorem \ref{mthm4}. We split the proof into several lemmas as follows.
Throughout the section we fix a compact $\RCD(K, N)$ space $(X, \dist, \meas)$ with $n=\dim_{\dist, \meas}(X)$ and a $(k+1)$-dimensional eigenmap $\Phi=(\phi_1, \ldots, \phi_{k+1}):X \to \mathbb{R}^{k+1}$ with
\begin{equation}\label{eq:1}
\Phi^*g_{\mathbb{R}^{k+1}}=g_X.
\end{equation}
Thanks to (\ref{eq:eigenest}), $\Phi$ is a $C_1$-Lipschitz map from $X$ to $\mathbb{R}^{k+1}$ for some $C_1 \in (0, \infty)$.

\begin{lemma}\label{houuy}
The following two conditions are equivalent:
\begin{enumerate}
\item $|\Phi|_{\mathbb{R}^{k+1}}$ is a constant function.
\item $(X, \dist, \mathcal{H}^n)$ is a non-collapsed $\RCD(K, n)$ space with 
\begin{equation}
\meas=\frac{\meas(X)}{\mathcal{H}^n(X)}\mathcal{H}^n.
\end{equation}
\end{enumerate}
\end{lemma}
\begin{proof}
Let us first prove the implication from (1) to (2).
Assume that $|\Phi|_{\mathbb{R}^{k+1}}$ is a constant function. Then Corollary \ref{cor:1} shows 
\begin{equation}
\nabla^*\Phi^*g_{\mathbb{R}^{k+1}}=-\frac{1}{4}\dist \Delta |\Phi|^2_{\mathbb{R}^{k+1}}=0.
\end{equation}
Thus $g \in D(\nabla^*)$ with $\nabla^*g=0$ which completes the proof of (2) because of Theorem \ref{thm:0}.

Next we prove the remaining implication. Assume that $(X, \dist, \mathcal{H}^n)$ is a non-collapsed $\RCD(K, n)$ space.
Then applying Theorem \ref{thm:0} again yields $\dist \Delta |\Phi|^2_{\mathbb{R}^{k+1}}=0$.
Since $\Delta |\Phi|^2_{\mathbb{R}^{k+1}} \in H^{1, 2}(X, \dist, \mathcal{H}^n)$, we see that $\Delta |\Phi|^2_{\mathbb{R}^{k+1}}$ is a constant function. In particular for any $f \in D(\Delta)$ with $\Delta f+\lambda f\equiv 0$ for some $\lambda \in (0, \infty)$, we have
\begin{equation}
\int_Xf\Delta |\Phi |^2_{\mathbb{R}^{k+1}}\di \meas =0
\end{equation}
which implies
\begin{equation}
\int_Xf|\Phi|^2_{\mathbb{R}^{k+1}}\di \meas=0.
\end{equation}
Since $f$ is arbitrary and the eigenfunctions form a Hilbert base of $L^2(X, \meas)$, we see that $|\Phi|^2_{\mathbb{R}^{k+1}}$ is also a constant function.
\end{proof}
From now on we assume that $|\Phi|$ is a constant function.
By rescaling, with no loss of generality we can assume $|\Phi|\equiv 1$.

\begin{lemma}\label{lemelemen}
We have $\min_i\lambda_i\le n$.
\end{lemma}
\begin{proof}
(\ref{eq:1}) yields
\begin{equation}\label{eq:11}
\langle \Phi^*g_{\mathbb{R}^{k+1}}, g\rangle =\langle g, g\rangle =n.
\end{equation}
Thus integrating (\ref{eq:11}) over $X$ shows
\begin{equation}
\min_i\lambda_i  \le \sum_{i=1}^{k+1}\frac{1}{\mathcal{H}^n(X)}\int_X|\dist \phi_i|^2\di \mathcal{H}^n=n.
\end{equation}
\end{proof}
\begin{lemma}\label{hoho}
For any $x \in X$ and any $\epsilon \in (0, 1)$ there exists $r \in (0, 1)$ such that $\Phi|_{B_{r}(x)}:B_r(x) \to (\Phi(B_r(x)), \dist_{\mathbb{S}^k})$ is a bi-Lipschitz map, that $\Phi|_{B_{r}(x)}$ is $(1+\epsilon)$-Lipschitz, and that $(\Phi|_{B_{r}(x)})^{-1}: (\Phi(B_r(x)), \dist_{\mathbb{S}^n}) \to (X, \dist)$ is $(1+\epsilon)$-Lipschitz. 
\end{lemma}
\begin{proof}
Fix a sufficiently small $\epsilon \in (0, 1)$.
Since $O(k+1)$ acts on $\mathbb{S}^k$ transitively, with no loss of generality we can assume that $\phi_{k+1}(x)=1$. Note that
\begin{equation}\label{eq:tt}
\sum_{j=1}^k\left(1+\frac{\phi_j^2}{1-\sum_{j=1}^k\phi_j^2}\right)\dist \phi_j\otimes \dist \phi_j+\sum_{j \neq l}^k\frac{\phi_j \phi_l}{1-\sum_{j=1}^k\phi_j^2}\dist \phi_j \otimes \dist \phi_l=g_{X}
\end{equation}
holds on a neighbourhood of $x$.
Thus we can find a sufficiently small $r_0 \in (0, \epsilon )$ satisfying that
\begin{equation}\label{pointwisezero}
\left| \sum_{j=1}^k\dist \phi_j \otimes \dist \phi_j-g_X\right|(y)<\epsilon
\end{equation}
for $\mathcal{H}^n$-a.e. $y \in B_{2r_0}(x)$. In particular for any $x \in B_{r_0}(y)$ and any $r \in (0, \delta r_0]$ we have
\begin{align}\label{ponthfgrb}
\frac{1}{\meas (B_r(y))}\int_{B_{\delta^{-1}r}(y)}\left|\sum_{j=1}^k\dist \phi_j \otimes \dist \phi_j-g_X\right|\di \meas
&<\epsilon\frac{\meas (B_{\delta^{-1}r}(y))}{\meas (B_r(y))} \nonumber \\
&\le \epsilon \int_0^{\delta^{-1}}\sinh^{N-1} t\di t \cdot \left(\int_0^1\sinh^{N-1} t\di t \right)^{-1} \nonumber \\
&=\sqrt{\epsilon},
\end{align}
where we used the Bishop-Gromov inequality for the rescaled space $(X, r^{-1}\dist, \meas)$ as an $\RCD(-(N-1), N)$ space, and $\delta=\delta(\epsilon, N) >\sqrt{\epsilon}$ is defined by satisfying
\begin{equation}
\sqrt{\epsilon}^{-1}=\int_0^{\delta^{-1}}\sinh^{N-1} t\di t \cdot \left(\int_0^1\sinh^{N-1} t\di t \right)^{-1}.
\end{equation}
Define a map $\hat{\Phi}$ from $X$ to $\mathbb{R}^k$ by
\begin{equation}
\hat{\Phi}:=(\phi_1, \phi_2, \ldots, \phi_k).
\end{equation}
Then (\ref{ponthfgrb}) says that $\hat{\Phi}$ is a locally uniformly $\delta$-isometric immersion on $B_{\delta r_0}(x)$.
Thus applying Theorem \ref{yhug} gives that there exists $r_1 \in (0, \delta r_0)$ such that $\hat{\Phi}|_{B_{r_1}(x)}$ is a $(1+\Psi(\epsilon; K, n, k, C_1))$-Lipschitz embedding map and that $(\hat{\Phi}|_{B_{r_1}(x)})^{-1}$ is $(1+\Psi(\epsilon; K, n, k, C_1))$-Lipschitz.

On the other hand recall that a map $\pi:=\pi_r$ from $(B_r(\Phi(x)), \dist_{\mathbb{S}^k})$ to $(\mathbb{R}^k, \dist_{\mathbb{R}^k})$ defined by
\begin{equation}
\pi (x_1, x_2, \ldots, x_{k+1}):=(x_1, x_2, \ldots, x_k)
\end{equation}
is a $1$-Lipschitz embedding satisfying that $(\pi |_{B_r(\Phi (x))})^{-1}$ is $(1+\Psi(r; k))$-Lipschitz. Therefore since $\Phi=\pi^{-1} \circ \hat{\Phi}$ holds on $B_{r_1}(x)$, we have the desired statement because $\epsilon$ is arbitrary.
\end{proof}

\begin{lemma}\label{uuuuu}
$\Phi:(X, \dist) \to (\mathbb{S}^k, \dist_{\mathbb{S}^k})$ is $1$-Lipschitz.
\end{lemma}
\begin{proof}
By Lemma \ref{hoho}, for any $1$-Lipschitz function $F$ on $\mathbb{S}^k$ and any $\epsilon \in (0, 1)$, we see that $\mathrm{lip}(F\circ \Phi)(x) \le 1+\epsilon$ holds for any $x \in X$. Therefore we see that $F\circ \Phi$ is also $1$-Lipschitz because $\epsilon$ is arbitrary, the local slope is an upper gradient and $(X, \dist)$ is a geodesic space (or we can also check this $1$-Lipschitz property directly via the Sobolev-to-Lipschitz property).
%where we used a fact that $|\nabla (F|_{\mathbb{S}^k(\sqrt{k+1})})|=|\nabla F|$ which comes from 
%\begin{equation}
%\frac{\dist_{\mathbb{S}^k}(x, y)}{\dist_{\mathbb{R}^{k+1}} (x, y)} \to 1
%\end{equation}
%as $y \to x$ in $\mathbb{S}^k(\sqrt{k+1})$. 
In particular for fixed $x \in X$, taking $F(z)=\dist_{\mathbb{S}^k}(\Phi(x), z)$ yields
\begin{equation}\label{eq:11112}
\dist_{\mathbb{S}^k}(F(\Phi(x)), F(\Phi(y))) =|F(\Phi(x))-F(\Phi(y))| \le \dist(x, y)
\end{equation}
which completes the proof.
\end{proof}
From now on we assume $k=n$.
\begin{lemma}\label{esse}
$\Phi:(X, \dist) \to (\mathbb{S}^n, \dist_{\mathbb{S}^n})$ is a local isometry.
\end{lemma}
\begin{proof}
Let us remark that $\Phi(X)$ is an open subset of $\mathbb{S}^n$ because $X$ is homeomorphic to an $n$-dimensional topological closed manifold by Theorem \ref{thm:bilip}, where we used the fact that if $F:M^n \to N^n$ is an injective continuous map between two $n$-dimensional topological manifolds $M^n$ and $N^n$, then $F$ is a local homeomorphism to $N^n$. Moreover since $\Phi(X)$ is compact, in particular, it is closed. Thus $\Phi$ is a surjective locally bi-Lipschitz map to $\mathbb{S}^n$.

%Note that for any $F\in \mathrm{Lip}(\mathbb{S}^n, \dist_{\mathbb{S}^n})$ there exists a sequence of $F_i \in C^{\infty}(\mathbb{S}^n)$ with $\sup_i\|\nabla F_i\|_{L^{\infty}}<\infty$ such that $F_i$ $H^{1,2}$-strongly converge to $F$ and that $\nabla F_i(x) \to \nabla F(x)$ holds for $\mathcal{H}^n$-a.e. $x \in \mathbb{S}^n$ (via the heat flow on $\mathbb{S}^n$. In particular $F_i \circ \Phi$ $H^{1, 2}$-strongly converge to $F\circ \Phi$).
%Combining this fact with Lemma \ref{locis} and Mcshane's lemma yields that for any $F \in \mathrm{Lip}(U, \dist_{\mathbb{S}^n})$ for an open subset $U$ of $\mathbb{S}^n$,
%\begin{equation}\label{oor}
%|\nabla (F\circ \Phi)|(x)=|\nabla F| (\Phi (x))
%\end{equation}
%holds for $\mathcal{H}^n$-a.e. $x \in \Phi^{-1}(U)$.

Using the compactness of $X$ with Lemma \ref{hoho}, find $r \in (0, 1)$ satisfying that $\Phi|_{B_{2r}(x)}$ is a bi-Lipschitz map to the image including $B_r(\Phi(x))$ for any $x \in X$. Fix $x \in X$ and let us define $f \in \mathrm{Lip}(B_r(\Phi(x)), \dist_{\mathbb{S}^n})$ by 
\begin{equation}
f(y):=\dist \left(x, (\Phi|_{B_{2r}(x)})^{-1}(y)\right).
\end{equation}
Then Lemma \ref{hoho} yields
\begin{equation}\label{eeeee}
\mathrm{lip}f(y) \le (1+\epsilon) \mathrm{lip} (f\circ \Phi)|( (\Phi|_{B_{2r}(x)})^{-1}(y)) \le 1+\epsilon
\end{equation}
holds for any $y \in B_r(\Phi(x))$ and any $\epsilon \in (0, 1)$. Since $B_r(\Phi(x))$ is a convex subset of $\mathbb{S}^n$, (\ref{eeeee}) shows that $f$ is $1$-Lipschitz on $B_r(\Phi(x))$ because $\epsilon$ is arbitrary. In particular we have for any $y \in B_r(\Phi(x))$
\begin{equation}
\dist(x, y)=|f(\Phi(x))-f(\Phi(y))|\le \dist_{\mathbb{S}^n}(\Phi(x), \Phi(y))
\end{equation}
which completes the proof because of Lemma \ref{uuuuu}.
\end{proof}
The following lemma completes the proof of Theorem \ref{mthm4}:

\begin{lemma}\label{complete}
$\Phi:(X, \dist) \to (\mathbb{S}^n, \dist_{\mathbb{S}^n})$ is an isometry.
\end{lemma}
\begin{proof}
Lemma \ref{esse} shows that $(X, \dist)$ is isometric to an $n$-dimensional closed Riemannian manifold whose sectional curvature is equal to $1$. Then Lemma \ref{lemelemen} with Obata's theorem yields that $(X, \dist)$ is isometric to $(\mathbb{S}^n, \dist_{\mathbb{S}^n})$. In particular $\Phi:(X, \dist) \to (\mathbb{S}^n, \dist_{\mathbb{S}^n})$ is an isometry because it is easy to see that if an eigenmap $\tilde{\Phi}:\mathbb{S}^n \to \mathbb{R}^{n+1}$ satisfies $\tilde{\Phi}^*g_{\mathbb{R}^{n+1}}=g_{\mathbb{S}^n}$, then $\tilde{\Phi}$ coincides with the canonical inclusion $\iota:\mathbb{S}^n \hookrightarrow \mathbb{R}^{n+1}$ up to multiplying an element of $O(n+1)$.
\end{proof}
\section{Sphere theorem}\label{6666r6r6r6}
The main purpose of this section is to prove Theorems \ref{mthm2} and \ref{thm:sphere}.
\subsection{Convergence result for eigenmaps}
The main result of this subsection is the following, which will play a key role in the proofs of Theorems \ref{mthm2} and \ref{thm:sphere}.
\begin{theorem}\label{qqw}
Let $(X_i, \dist_i, \mathcal{H}^n)$ be a sequence of compact non-collapsed $\RCD(K, n)$ spaces, let $k \in \mathbb{N}$ and let $\Phi_i:X_i \to \mathbb{R}^{k}$ be a $k$-dimensional eigenmap. Assume that
\begin{equation}\label{err1}
\frac{1}{\mathcal{H}^n(X_i)}\int_{X_i}\left|\Phi_i^*g_{\mathbb{R}^{k}}-g_{X_i}\right| \di \mathcal{H}^n \to 0
\end{equation}
and 
\begin{equation}
\sup_i\left(\mathrm{diam}(X_i, \dist_i)+\||\Phi_i|_{\mathbb{R}^k}\|_{L^2} + \Lambda(\Phi_i)\right)<\infty
\end{equation}
are satisfied.
Then after passing to a subsequence, there exist a compact non-collapsed $\RCD(K, n)$ space $(X, \dist, \mathcal{H}^n)$ and a $k$-dimensional eigenmap $\Phi:X \to \mathbb{R}^k$ such that the following holds.
\begin{enumerate}
\item $(X_{i}, \dist_{i}, \mathcal{H}^n)$ mGH-converge to $(X, \dist, \mathcal{H}^n)$. 
\item $X_{i}$ is homeomorphic to $X$ for any sufficiently large $i$. 
\item $\Phi_{i}$ converge to $\Phi$.
\item $\Phi^*g_{\mathbb{R}^k}=g_X$.
\end{enumerate}
\end{theorem}
\begin{proof}
By Theorem \ref{hohohonoho}, with no loss of generality we can assume that there exists a compact $\RCD(K, n)$ space $(X, \dist, \meas)$ such that  
\begin{equation}
\left(X_i, \dist_i, \mathcal{H}^n(X_i)^{-1}\mathcal{H}^n\right) \stackrel{\mathrm{mGH}}{\to} (X, \dist, \meas) 
\end{equation} 
holds.   

Applying Proposition \ref{ssx}, after passing to a subsequence, there exists a $k$-dimensional eigenmap $\Phi:X \to \mathbb{R}^k$ such that $\Phi_i$ converge to $\Phi$.
By Proposition \ref{weakriem} with (\ref{err1}), we see that $g_{X_i}$ $L^2$-strongly converge to $g_X$. In particular we have
\begin{equation}
\int_{X}|\Phi^*g_{\mathbb{R}^{k}}-g_X|^2\di \meas = \lim_{i \to \infty}\frac{1}{\mathcal{H}^n(X_i)}\int_{X_i}|\Phi^*_ig_{\mathbb{R}^{k}}-g_{X_i}|^2\di \mathcal{H}^n=0,
\end{equation}
(thus (4) of Theorem \ref{qqw} holds) and
\begin{equation}\label{aaae}
\dim_{\dist, \meas}(X)=\frac{1}{\meas (X)}\int_X|g_X|^2\di \meas= \lim_{i \to \infty}\frac{1}{\mathcal{H}^n(X_i)}\int_{X_i}|g_{X_i}|^2\di \mathcal{H}^n=\lim_{i \to \infty}\dim_{\dist_i, \meas_i}(X_i)=n.
\end{equation}
Therefore Theorem \ref{wnon} (or Theorem \ref{GHmGH} with \cite[Thm.1.4]{DePhillippisGigli}) yields that $(X, \dist, \mathcal{H}^n)$ is a non-collapsed $\RCD(K, n)$ space with $\mathcal{H}^n(X)\meas=\meas (X) \mathcal{H}^n$. Then the sequence $\{(X_i, \dist_i)\}_i$ is uniformly Reifenberg flat, that is:
\begin{itemize}
\item[$(\spadesuit)$] For any $\epsilon \in (0, 1)$ there exist $i_0 \in \mathbb{N}$ and $r_0 \in (0, \epsilon)$
such that for any $i \in \mathbb{N}_{\ge i_0}$, any $x_i \in X_i$ and any $r \in (0, r_0)$ we have $\dist_{\mathrm{GH}}(B_r(x_i), B_r(0_n))<\epsilon r$.
\end{itemize}
Although the proof of $(\spadesuit)$ is well-known, for reader's convenience, let us give a proof (see also \cite[Rem.2.4]{HondaMondello}).

Fix $\epsilon \in (0, 1)$. Thanks to Theorem \ref{thm:bilip}, $X$ is Reifenberg flat. Thus there exists $r_0 \in (0, \epsilon)$ such that 
\begin{equation}
\dist_{\mathrm{GH}}(B_r(x), B_r(0_n))<\epsilon r,\quad \forall r \in (0, 2r_0],\,\forall x \in X
\end{equation} 
holds. In particular there exists $i_0 \in \mathbb{N}$ such that for any $i \in \mathbb{N}_{\ge i_0}$ and any $x_i \in X_i$ we have
\begin{equation}
\dist_{\mathrm{GH}}(B_{2r_0}(x_i), B_{2r_0}(0_n))<4\epsilon r_0.
\end{equation}
Fix $i \in \mathbb{N}_{\ge i_0}$ and $x_i \in X_i$.
Therefore Theorem \ref{thm:bishop2} yields
\begin{equation}
1 - \Psi(\epsilon; K, n)\le \frac{\mathcal{H}^n(B_{2r_0}(x_i))}{\mathcal{H}^n(B_{2r_0}(0_n))} \le 1 + \Psi(\epsilon; K, n).
\end{equation}
The Bishop-Gromov and Bishop inequalities (Theorem \ref{thm:bishop}) yield that 
\begin{equation}
1 - \Psi(\epsilon; K, n) \le \frac{\mathcal{H}^n(B_{2r}(x_i))}{\mathcal{H}^n(B_{2r}(0_n))}\le 1 + \Psi(\epsilon; K, n),\quad \forall r \in (0, r_0]
\end{equation}
holds. Thus applying Theorem \ref{thm:bishop2} again shows
\begin{equation}
\dist_{\mathrm{GH}}(B_r(x_i), B_r(0_n))<\Psi(\epsilon; K, n)r,\quad \forall r \in (0, r_0]
\end{equation}
which completes the proof of $(\spadesuit)$ because $\epsilon$ is arbitrary.

Then the remaining statement, that $X_{i}$ is homeomorphic to $X$ for any sufficiently large $i$, follows from $(\spadesuit)$ with the topological stability theorem proved in \cite[Thm.A.1.2 and A.1.3]{CheegerColding1}.
\end{proof}
\subsection{Proof of Theorems \ref{mthm2} and \ref{thm:sphere}}
Let us conclude the proofs of Theorems \ref{mthm2} and \ref{thm:sphere}.

\textit{Proof of Theorem \ref{mthm2}.}

Let us prove (1) of Theorem \ref{mthm2}.
The proof is done by contradiction.
If not, there exist $\epsilon_0 \in (0, 1)$, a sequence of compact non-collapsed $\RCD(K, n)$ spaces $(X_i, \dist_i, \mathcal{H}^n)$ with $\mathrm{diam}(X_i, \dist_i) \le d$, and a sequence of irreducible $(n+1)$-dimensional eigenmaps $\Phi_i:X_i \to \mathbb{R}^{n+1}$ with $L(\Phi_i) \ge \tau$ such that  
\begin{equation}\label{err}
\frac{1}{\mathcal{H}^n(X_i)}\int_{X_i}\left|\Phi_{i}^*g_{\mathbb{R}^{n+1}}-g_{X_i}\right|^2\di \mathcal{H}^n \to 0
\end{equation}
and 
\begin{equation}\label{cc}
\dist_{\mathrm{GH}}(X_i, \mathbb{S}^n(a_i)) \ge \epsilon_0
\end{equation}
are satisfied, where 
\begin{equation}
a_i^2=\frac{1}{\mathcal{H}^n(X_i)}\int_{X_i}|\Phi_i|^2\di \mathcal{H}^n.
\end{equation}
Applying Theorem \ref{qqw} with Corollary \ref{ssx2} shows that after passing to a subsequence, there exist a compact non-collapsed $\RCD(K, n)$ space and an irreducible $(n+1)$-dimensional eigenmap $\Phi:X \to \mathbb{R}^{n+1}$ such that $(X_{i}, \dist_{i}, \mathcal{H}^n)$ mGH-converge to $(X, \dist, \mathcal{H}^n)$, that $\Phi_i$ converge to $\Phi$ and that $\Phi^*g_{\mathbb{R}^{n+1}}=g_X$.
Then Theorem \ref{mthm4} yields that $|\Phi|$ is a constant function and that $(X, \dist)$ is isometric to $(\mathbb{S}^n(|\Phi|), \dist_{\mathbb{S}^n(|\Phi|)})$ which contradicts (\ref{cc}) because of $a_i^2 \to |\Phi|^2$.

The proof of (2) of Theorem \ref{mthm2} is similarly done by contradiction after establishing the following result:
\begin{itemize}
\item[$(\heartsuit)$]If a sequence of non-collapsed $\RCD(K, n)$ spaces $(X_i, \dist, \mathcal{H}^n)$ satisfies $\dist_{\mathrm{GH}}(X_i, \mathbb{S}^n(a)) \to 0$ for some $a \in (0, \infty)$, then there exists a sequence of irreducible $(n+1)$-dimensional eigenmaps $\Phi_i:X_i \to \mathbb{R}^{n+1}$ such that $\Phi_i$ converge to the canonical inclusion $\iota:\mathbb{S}^n(a) \to \mathbb{R}^{n+1}$ with $L(\Phi_i) \to L(\iota)$.
\end{itemize}
However $(\heartsuit)$ is a direct consequence of Proposition \ref{prop:appl} with Theorem \ref{GHmGH}. Thus we conclude. $\,\,\,\,\,\square$

\textit{Proof of Theorem \ref{thm:sphere}.}

The proof is also done by contradiction. If not, then we can find a sequence of compact non-collapsed $\RCD(K, n)$ spaces $(X_i, \dist_i, \mathcal{H}^n)$ with $\mathrm{diam} (X_i, \dist_i) \le d$, and a sequence of irreducible $(n+1)$-dimensional eigenmaps $\Phi_i:X_i \to \mathbb{R}^{n+1}$ with $L(\Phi_i) \ge \tau$ such that $X_i$ is not homeomorphic to $\mathbb{S}^n$ and that
\begin{equation}
\frac{1}{\mathcal{H}^n(X_i)}\int_{X_i}|\Phi^*_ig_{\mathbb{R}^{n+1}}-g_{X_i}|\di \mathcal{H}^n\to 0
\end{equation} 
holds. Then applying Theorems \ref{mthm2} and \ref{qqw} with Corollary \ref{ssx2}, after passing to a subsequence, we see that $(X_i, \dist_i)$ GH-converge to $(\mathbb{S}^n(a), \dist_{\mathbb{S}^n(a)})$ for some $a \in (0, \infty)$ and that $X_i$ is homeomorphic to $\mathbb{S}^n(a)$ for any sufficiently large $i$. This is a contradiction. $\,\,\,\,\,\square$

\section{Topological finiteness theorems}
In this section we establish topological finiteness theorems for almost isometric immersions. 
%Since the proofs of the main results in this section (Theorems \ref{topological1} and \ref{topological2}) are essentially same to that of Theorems \ref{mthm3} and Theorem \ref{thm:sphere} based on Theorem \ref{thm:bilip}, we give only sketches of the proofs.
\subsection{Convergence result for regular maps}
Let us fix $k \in \mathbb{N}$ and a mGH-convergent sequence of compact $\RCD(K, N)$ spaces:
\begin{equation}
(X_i, \dist_i, \meas_i) \stackrel{\mathrm{mGH}}{\to} (X, \dist, \meas). 
\end{equation}
Note that if a sequence $\phi_i \in D(\Delta_i)$ satisfies
$\sup_i\|\Delta_i \phi_i\|_{L^{\infty}}<\infty$, then
$\sup_i\|\nabla_i \phi_i\|_{L^{\infty}}<\infty$ holds, which is a direct consequence of regularity results proved in \cite[Thm.3.1]{AMS} and \cite[Thm.3.1]{Jiang}, independently. Since we will use this fact immediately below, for reader's convenience, we give a sketch of the proof as follows.

From \cite[Thm.3.1]{AMS} and \cite[Thm.3.1]{Jiang}, it is enough to check that under the assumption $\int_{X_i}\phi_i\di \meas_i=0$ with $\meas_i(X_i)=1$,
\begin{equation}\label{eq:fdrrrefer}
\sup_i\|\phi_i\|_{L^{\infty}(X_i)}<\infty.
\end{equation}
Note that the Poincar\'e inequality implies
\begin{align*}
\int_{X_i}\left|\phi_i-\int_{X_i}\phi_i\di\meas_i\right|^2\di \meas_i&\le C(K, N, \mathrm{diam}(X_i, \dist_i))\int_{X_i}|\nabla \phi_i|^2\di \meas_i \\
&=-C(K, N, \mathrm{diam}(X_i, \dist_i))\int_{X_i}\phi_i\Delta_i\phi_i\di \meas_i \\
&\le C(K, N, \mathrm{diam}(X_i, \dist_i))\|\Delta_i\phi_i\|_{L^{\infty}}\int_{X_i}|\phi_i|\di \meas_i.
\end{align*}
In particular applying the Cauchy-Schwarz inequality shows
\begin{equation}\label{eq:hyt}
\sup_i \|\phi_i\|_{L^2}<\infty.
\end{equation}
Take a ball $B_{\epsilon}(x_i)$ in $X_i$ with $\meas_i(X_i \setminus B_{\epsilon}(x_i))>0$, and find $\hat{\phi}_i \in H^{1, 2}_0(B_{\epsilon}(x_i), \dist_i, \meas_i) \cap D(\Delta_i, B_{\epsilon}(x_i))$ with $\Delta_i \hat{\phi}_i=\Delta_i \phi_i$ and 
\begin{equation}\label{eq:knnnjuiu}
\|\hat{\phi}_i\|_{L^{\infty}(B_{\epsilon}(x_i))}\le C(K, N)\epsilon^2\|\Delta \phi_i\|_{L^{\infty}}
\end{equation}
(see also \cite[Thm.4.1]{BMosco}). Then applying \cite[Lem.4.1]{Jiang} for a harmonic function $\phi_i-\hat{\phi}_i$ yields
\begin{equation}\label{eqeww}
\|\phi_i-\hat{\phi}_i\|_{L^{\infty}(B_{\epsilon/2}(x_i))}\le C(K, N, \mathrm{diam}(X_i, \dist_i), \epsilon)\frac{1}{\meas_i(B_{\epsilon}(x_i))}\int_{B_{\epsilon}(x_i)}|\phi_i-\hat{\phi}_i|\di \meas_i.
\end{equation}
In particular combining (\ref{eq:hyt}) with (\ref{eq:knnnjuiu}) and (\ref{eqeww}) yields (\ref{eq:fdrrrefer}) via taking $\epsilon/2$-nets.

As in Definition \ref{def:conveigen} for eigenmaps, let us define the convergence of Lipschitz (not necessary eigen) maps as follows.  
\begin{definition}[Convergence of Lipschitz maps]\label{def:convgeneral}
We say that a sequence of Lipschitz maps $\Phi_i:X_i \to \mathbb{R}^k$ \textit{converge} to a Lipschitz map $\Phi:X \to \mathbb{R}^k$ if $\pi_j\circ \Phi_i$ $L^2$-strongly converge to $\pi_j\circ \Phi$ for any $j$, where $\pi_j:\mathbb{R}^k \to \mathbb{R}$ is the projection to the $j$-th $\mathbb{R}$, and $\Phi_i^*g_{\mathbb{R}^k}$ $L^2$-strongly converge to $\Phi^*g_{\mathbb{R}^k}$ on $X$.
\end{definition}
\begin{proposition}[Compactness for regular maps]\label{compactne}
Let $\Phi_i=(\phi_{1, i},\ldots, \phi_{k, i}) :X_i \to \mathbb{R}^k$ be a sequence of regular maps with
\begin{equation}\label{eq:33wwssddee}
\sup_{j, i}\left( \|\phi_{j, i}\|_{L^2}+\|\Delta_i \phi_{j, i}\|_{L^{\infty}}\right)<\infty.
\end{equation}
Then after passing to a subsequence there exists a regular map $\Phi:X \to \mathbb{R}^k$ such that $\Phi_i$ converge to $\Phi$.
\end{proposition}
\begin{proof}
From the argument at the beginning of this subsection with (\ref{eq:33wwssddee}), we know $\sup_{j, i} \|\nabla_i\phi_{j, i}\|_{L^{\infty}}<\infty$. Then the desired conclusion follows from an argument similar to the proof of Proposition \ref{ssx}.
\end{proof}
Let us introduce the following ``gap theorems'':
\begin{theorem}[Gap theorem I]\label{mthm5}
For any $K \in \mathbb{R}$, any $N \in [1, \infty)$ and all $d, \lambda \in (0, \infty)$ there exists $\epsilon_0:=\epsilon_0(K, N, d, \lambda) \in (0, 1)$ such that if a compact $\RCD(K, N)$ space $(X, \dist, \meas)$ with $n=\dim_{\dist, \meas}(X)$ satisfies $\mathrm{diam}(X, \dist)\le d$, then 
\begin{equation}
\frac{1}{\meas (X)}\int_X|\Phi^*g_{\mathbb{R}^k}-g_X|\di \meas \ge \epsilon_0
\end{equation}
holds for any $k \in \mathbb{N}_{\le n}$ and any regular map $\Phi=(\phi_1,\ldots, \phi_k):X \to \mathbb{R}^k$ with $\|\Delta \phi_i\|_{L^{\infty}}\le \lambda$.
\end{theorem}
\begin{proof}
The proof is done by contradiction. If not, there exist $n \in \mathbb{N}$, $k \in \mathbb{N}_{\le n}$, a sequence of compact $\RCD(K, N)$ spaces $(X_i, \dist_i, \meas_i)$, and a sequence of regular maps $\Phi_i=(\phi_{1, i},\ldots, \phi_{k, i}):X_i \to \mathbb{R}^k$ such that $\|\Delta_i\phi_{j, i}\|_{L^{\infty}}\le \lambda$,  $\sup_{j, i}\|\phi_{j, i}\|_{L^2}<\infty$,  $\mathrm{diam}(X_i, \dist_i)\le d$,  $\meas_i(X_i)=1$, $\dim_{\dist_i, \meas_i}(X_i)=n$ and 
\begin{equation}\label{hohohomo}
\int_X|\Phi^*_ig_{\mathbb{R}^k}-g_{X_i}|\di \meas_i \to 0
\end{equation}
are satisfied.
Thus Proposition \ref{compactne} yields that after passing to a subsequence, there exists a regular map $\Phi:X \to \mathbb{R}^k$ such that $\Phi_i$ converge to $\Phi$. Therefore (\ref{hohohomo}) with Proposition \ref{weakriem} shows that $\Phi^*g_{\mathbb{R}^k}=g_X$ holds and that $g_{X_i}$ $L^2$-strongly converge to $g_X$ on $X$.
In particular
\begin{equation}
\dim_{\dist, \meas}(X)=\int_X|g_X|^2\di \meas=\lim_{i \to \infty}\int_{X_i}|g_{X_i}|^2\di \meas_i=\lim_{i \to \infty}\dim_{\dist_i, \meas_i}(X_i)=n.
\end{equation}
Therefore (2) of Theorem \ref{thm:bilip} yields that $X$ is homeomorphic to an $n$-dimensional closed manifold. In particular there is no local homeomorphism from $X$ into $\mathbb{R}^k$, which contradicts (1) of Theorem \ref{thm:bilip}. 
\end{proof}
%It is worth pointing out that the following is a direct consequence of Theorem \ref{yhug} with Propositions \ref{prop:low} and \ref{compactne}, where we omit the proof because it is easy to check this by contradiction.
\begin{theorem}[Gap theorem II]\label{mthm5900000}
For any $K \in \mathbb{R}$ and any $N \in [1, \infty)$ there exists $\epsilon_0:=\epsilon_0(K, N) \in (0, 1)$ such that we have 
\begin{equation}
\|\Phi^*g_{\mathbb{R}^k}-g_X\|_{L^{\infty}} \ge \epsilon_0
\end{equation}
for any compact $\RCD(K, N)$ space $(X, \dist, \meas)$ with $n=\dim_{\dist, \meas}(X)$, any $k \in \mathbb{N}_{\le n}$ and any regular map $\Phi=(\phi_1,\ldots, \phi_k):X \to \mathbb{R}^k$.
\end{theorem}
\begin{proof}
Assume that $\|\Phi^*g_{\mathbb{R}^k}-g_X\|_{L^{\infty}}$ is sufficiently small depending only on $K, N$. Then we have $\|\nabla \phi_i\|_{L^{\infty}}\le 2n$ because 
\begin{equation}
\sum_i|\nabla \phi_i|^2=\langle \Phi^*g_{\mathbb{R}^k}, g_X\rangle =\langle \Phi^*g_{\mathbb{R}^k}-g_X, g_X \rangle +|g_X|^2 \le \sqrt{n}\|\Phi^*g_{\mathbb{R}^k}-g_X\|_{L^{\infty}}+n.
\end{equation}
Thus (1) of Theorem \ref{yhug} with \cite[Thm.A.1.2]{CheegerColding1} implies that $X$ is homeomorphic to an $n$-dimensional closed manifold. In particular there is no local homeomorphism from $X$ into $\mathbb{R}^k$, which contradicts (2) of Theorem \ref{yhug}.
\end{proof}
%Using Proposition \ref{appp} instead of Theorem \ref{spectral2}, by an argument similar to the proof of Proposition %\ref{prop:appl}, we can get the following:
%\begin{proposition}[Approximation]\label{prop:appl2}
%Let $\hat{\Phi}=(\hat{\phi}_1, \ldots, \hat{\phi}_k):X \to \mathbb{R}^k$ be a regular Lipschitz map. Then there exists a sequence of regular Lipschitz maps $\hat{\Phi}_i=(\hat{\phi}_{1, i},\ldots, \hat{\phi}_{k, i}):X_i \to \mathbb{R}^k$ such that $\hat{\Phi}_i$ converge to $\hat{\Phi}$ and that 
%\begin{equation}
%\sup_{j, i}\left(\|\nabla \hat{\phi}_{j, i}\|_{L^{\infty}} +\|\Delta_i\hat{\phi}_{j, i}\|_{L^{\infty}}\right)<\infty
%\end{equation}
%holds.
%\end{proposition}
\subsection{Topological finiteness theorems}\label{finiteness}
We are now in a position to give two topological finiteness theorems.
The first one is stated for any compact (not necessary non-collapsed) $\RCD(K, N)$ spaces as follows:
\begin{theorem}[Topological finiteness theorem I]\label{topological1}
For any $K \in \mathbb{R}$, any $k \in \mathbb{N}$, any $N \in [1, \infty)$ and all $d, \lambda \in (0, \infty)$, there exists $\delta:=\delta(K, k, N, d, \lambda) \in (0, 1)$ such that the following holds.
Let us denote by $\mathcal{M}:=\mathcal{M}(K, k, N, d, \lambda)$ the set of (isometry classes of) compact $\RCD(K, N)$ space $(X, \dist, \meas)$ satisfying that $\mathrm{diam} (X, \dist)\le d$ and that there exists a regular map $\Phi=(\phi_1,\ldots, \phi_k):X \to \mathbb{R}^{k}$ such that $\|\Delta \phi_i\|_{L^{\infty}}\le \lambda$ and
\begin{equation}\label{nbggh}
\|\Phi^*g_{\mathbb{R}^{k}}-g_X\|_{L^{\infty}}\le \delta
\end{equation}
are satisfied. Then for any $(X, \dist, \meas) \in \mathcal{M}$, $X$ is homeomorphic to an $n$-dimensional topological closed manifold. Moreover $\mathcal{M}$ has finitely many members up to homeomorphism.
\end{theorem}
\begin{proof}
Since the proof of the first statement is quite similar to that of Theorems \ref{mthm2} and \ref{thm:sphere}, we give a proof only for the topological finiteness statement. 

If not, then we can find a sequence of compact $\RCD(K, N)$ spaces $(X_i, \dist_i, \meas_i)$ and a sequence of regular maps $\Phi_i=(\phi_{1, i},\ldots, \phi_{k, i}):X_i \to \mathbb{R}^{k}$ such that $\mathrm{diam} (X_i, \dist_i) \le d$ holds, that $\meas_i(X_i)=1$ holds, that $\|\Delta_i\phi_{j, i}\|_{L^{\infty}} \le \lambda$ holds, that $\sup_{j, i}\|\phi_{j, i}\|_{L^{2}}<\infty$ holds, that $X_i$ is not homeomorphic to $X_j$ for all $i \neq j$, and that
\begin{equation}\label{99908}
\|\Phi^*_ig_{\mathbb{R}^{k}}-g_{X_i}\|_{L^{\infty}}\to 0
\end{equation} 
holds. 
Applying Theorem \ref{hohohonoho} with Proposition \ref{compactne} shows that after passing to a subsequence, there exist a compact $\RCD(K, N)$ space $(X, \dist, \meas)$ and a regular Lipschitz map $\Phi:X \to \mathbb{R}^k$ such that $(X_i, \dist_i, \meas_i)$ mGH-converge to $(X, \dist, \meas)$ and that $\Phi_i$ converge to $\Phi$. In particular (\ref{99908}) with Proposition \ref{weakriem} yields that $g_{X_i}$ $L^2$-strongly converge to $g_X$ on $X$ and that
\begin{equation}\label{qqwweerr}
\Phi^*g_{\mathbb{R}^k}=g_X
\end{equation}
holds. Thus Theorem \ref{thm:bilip} yields that $X$ is Reifenberg flat. 
With no loss of generality we can assume that $\dim_{\dist_i, \meas_i}(X_i)$ does not depend on $i$, thus we denote it by $n$. 
Then we have
\begin{equation}\label{referefer}
\dim_{\dist, \meas}(X)=\int_X|g_X|^2\di \meas=\lim_{i \to \infty}\int_{X_i}|g_{X_i}|^2\di \meas_i=\lim_{i \to \infty}\dim_{\dist_i, \meas_i}(X_i)=n.
\end{equation}

On the other hand applying Theorem \ref{highregular} with (\ref{99908}) (see also the proof of Theorem \ref{yhug}) shows that the sequence of $\{(X_i, \dist_i)\}_i$ is uniformly Reifenberg flat.
Therefore applying the topological stability theorem proved in \cite[Thm.A.1.2 and A.1.3]{CheegerColding1} with the Reifenberg flatness of $X$ and (\ref{referefer}) yields that  $X$ is homeomorphic to $X_i$ for any sufficiently large $i$, which is a contradiction.
\end{proof}
The second one is stated for non-collapsed $\RCD(K, n)$ spaces, however the assumption (\ref{hytfff}) is weaker than (\ref{nbggh}):
\begin{theorem}[Topological finiteness theorem II]\label{topological2}
For any $K \in \mathbb{R}$, any $k, n \in \mathbb{N}$ and all $d, \lambda \in (0, \infty)$, there exists $\delta:=\delta(K, k, n, d, \lambda) \in (0, 1)$ such that the following holds.
Let us denote by $\mathcal{M}:=\mathcal{M}(K, k, n, d, \lambda)$ the set of (isometry classes of) compact non-collapsed $\RCD(K, n)$ space $(X, \dist, \mathcal{H}^n)$ satisfying that $\mathrm{diam} (X, \dist)\le d$ and that there exists a regular map $\Phi=(\phi_1,\ldots, \phi_k):X \to \mathbb{R}^{k}$ such that $\|\Delta \phi_i\|_{L^{\infty}}\le \lambda$ and
\begin{equation}\label{hytfff}
\frac{1}{\mathcal{H}^n(X)}\int_X|\Phi^*g_{\mathbb{R}^{k}}-g_X|\di \mathcal{H}^n\le \delta
\end{equation}
are satisfied. Then for any $(X, \dist, \mathcal{H}^n) \in \mathcal{M}$, $X$ is homeomorphic to an $n$-dimensional topological closed manifold. Moreover $\mathcal{M}$ has finitely many members up to homeomorphism.
\end{theorem}
\begin{proof}
The proof is quite similar to that of Theorem \ref{topological1}. The only different part is in the proof of the uniform Reifenberg flatness. However this is justified by the standard way using the non-collapsed condition as already done in the proof of Theorem \ref{qqw}. Thus we omit it.
\end{proof}

\end{document}